\documentclass[10pt]{amsart}
\usepackage{geometry}
\usepackage[dvips]{graphicx}

\usepackage{amscd,amsfonts,amssymb,amsmath,amsthm,latexsym}
\usepackage{mathrsfs}

\usepackage{graphics}
\usepackage[all]{xy}
\xyoption{curve}
\xyoption{import}
\xyoption{arc}
\xyoption{ps}

\numberwithin{equation}{section}

\theoremstyle{plain}
    \newtheorem{thm}{Theorem}[section]
    \newtheorem{lemma}[thm]{Lemma}
    \newtheorem{coro}[thm]{Corollary}
    \newtheorem{prop}[thm]{Proposition}

    \newtheorem{question}[thm]{Question}
\theoremstyle{definition}
    \newtheorem{defi}[thm]{Definition}
    \newtheorem{ex}[thm]{Example}
    \newtheorem{remark}[thm]{Remark}
\theoremstyle{remark}

\newcommand{\semifield}{\mathbb{P}}
\newcommand{\Trop}{\operatorname{Trop}}
\newcommand{\Z}{\mathbb{Z}}
\newcommand{\Q}{\mathbb{Q}}
\newcommand{\myFF}{\mathcal{F}}
\newcommand{\x}{\mathbf{x}}

\newcommand{\TT}{\mathbb{T}_n}
\newcommand{\myAA}{\mathcal{A}}
\newcommand{\X}{\mathbf{X}}
\newcommand{\z}{z}
\newcommand{\g}{\mathbf{g}}
\newcommand{\e}{\mathbf{e}}

\newcommand{\suchthat}{\ | \ }
\newcommand{\surf}{(\Sigma,\mathbb{M})}
\newcommand{\qtau}{Q(\tau)}
\newcommand{\unredqtau}{\widehat{Q}(\tau)}
\newcommand{\qsigma}{Q(\sigma)}
\newcommand{\unredstau}{\widehat{S}(\tau)}

\newcommand{\stau}{S(\tau)}

\newcommand{\ssigma}{S(\sigma)}
\newcommand{\unredqstau}{(\widehat{Q}(\tau),\widehat{S}(\tau))}
\newcommand{\qstau}{(Q(\tau),S(\tau))}
\newcommand{\qssigma}{(Q(\sigma),S(\sigma))}

\newcommand{\arc}{i}
\newcommand{\jacobqs}{\mathcal{P}(Q,S)}
\newcommand{\jacobqstau}{{\mathcal{P}(Q(\tau),S(\tau))}}

\newcommand{\surfnoM}{\Sigma}
\newcommand{\marked}{\mathbb{M}}
\newcommand{\punct}{\mathbf{P}}
\newcommand{\arcsinsurf}{\mathbf{A}^\circ(\Sigma,\mathbb{M})}
\newcommand{\taggedinsurf}{\mathbf{A}^{\bowtie}(\Sigma,\mathbb{M})}
\newcommand{\tagfunction}{\mathfrak{t}}
\newcommand{\surfpone}{(\Sigma,\marked\cup\punct_1)}
\newcommand{\surfpn}{(\Sigma,\marked\cup\punct_n)}

\newcommand{\field}{K}
\newcommand{\C}{\mathbb{C}}
\newcommand{\completeRQ}{R\langle\langle Q\rangle\rangle}
\newcommand{\usualRQ}{R\langle Q\rangle}
\newcommand{\idealM}{\mathfrak{m}}
\newcommand{\pot}{\completeRQ_{\operatorname{cyc}}}
\newcommand{\qsreduced}{(Q_{\operatorname{red}},S_{\operatorname{red}})}
\newcommand{\qstrivial}{(Q_{\operatorname{triv}},S_{\operatorname{triv}})}
\newcommand{\muti}{\mu_i}
\newcommand{\premuti}{\widetilde{\mu}_i}
\newcommand{\al}{\mathfrak{a}}
\newcommand{\be}{\mathfrak{b}}
\newcommand{\ga}{\mathfrak{c}}
\newcommand{\rh}{\mathfrak{r}}
\newcommand{\si}{\mathfrak{s}}
\newcommand{\pipi}{\mathfrak{p}}
\newcommand{\io}{\mathfrak{i}}
\newcommand{\image}{\operatorname{im}}
\newcommand{\ored}{\operatorname{red}}
\newcommand{\oin}{\operatorname{in}}
\newcommand{\oout}{\operatorname{out}}
\newcommand{\Gr}{\operatorname{Gr}}
\newcommand{\Hom}{\operatorname{Hom}}

\begin{document}

\title[QPs associated to triangulated surfaces III]{Quivers with potentials associated to triangulated surfaces, Part III: tagged triangulations and cluster monomials}
\author[Giovanni Cerulli Irelli and Daniel Labardini-Fragoso]{Giovanni Cerulli Irelli\\ Daniel Labardini-Fragoso}
\address{Mathematisches Institut, Universit\"{a}t Bonn, Bonn, Germany 53115}
\email{cerulli@math.uni-bonn.de}
\email{labardini@math.uni-bonn.de}
\maketitle

\begin{abstract} To each tagged triangulation of a surface with marked points and non-empty boundary we associate a quiver with potential, in such a way that whenever we apply a flip to a tagged triangulation, the Jacobian algebra of the QP associated to the resulting tagged triangulation is isomorphic to the Jacobian algebra of the QP obtained by mutating the QP of the original one.  Furthermore, we show that any two tagged triangulations are related by a sequence of flips compatible with QP-mutation. We also prove that for each of the QPs constructed, the ideal of the non-completed path algebra generated by the cyclic derivatives is admissible and the corresponding quotient is isomorphic to the Jacobian algebra. These results, which generalize some of the second author's previous work for ideal triangulations, are then applied to prove properties of cluster monomials, like linear independence, in the cluster algebra associated to the given surface by Fomin-Shapiro-Thurston (with an arbitrary system of coefficients).
\end{abstract}

\tableofcontents

\section{Introduction}\label{sec:intro}

In \cite{Lqps}, the second author associated a quiver with potential (QP for short) to each ideal triangulation of a bordered surface with marked points, then proved that flips of arcs correspond to QP-mutations, and that, provided the surface has non-empty boundary, the associated QPs are non-degenerate and Jacobi-finite. However, the definition of the QPs that should correspond to tagged triangulations was not given, mainly because it was not clear that the ``obvious" potentials would possess the flip $\leftrightarrow$ QP-mutation compatibility that was proved for the QPs associated to ideal triangulations. In this paper we show that the Jacobian algebras of these ``obvious" potentials indeed possess the desired compatibility as long as the underlying surface has non-empty boundary. 
Then we show that the Jacobian algebras of these QPs can be obtained without the need of completion, that is, that each of them is (canonically) isomorphic to the quotient of the (non-completed) path algebra by the ideal generated by the cyclic derivatives, this ideal being admissible (in the classical sense of representation theory of associative algebras). The latter fact has nice consequences for the cluster algebra of the surface, for then we can use Derksen-Weyman-Zelevinsky's far-reaching homological interpretation of the $E$-invariant and the representation-theoretic interpretations of $F$-polynomials and $\g$-vectors, to obtain information about cluster monomials and their Laurent expansions with respect to a fixed cluster.

Let us describe the contents of the paper in more detail. We open with Section \ref{sec:background}, where we review the basic background on cluster algebras, quivers with potentials (and their mutations), decorated representations (and their mutations), and the relation of cluster algebras with the representation theory of quivers with potentials. In Section \ref{sec:triangulations} we recall the definition and basic properties of tagged triangulations of surfaces (and their flips). Sections \ref{sec:background} and \ref{sec:triangulations} contain the statements of most of the facts from the mentioned subjects that we use for our results in latter sections. Our intention is to be as self-contained as possible, but we also try to be as concise as possible.

In Section \ref{sec:QPs-of-tagged-triangulations} we give the definition of the QP $\qstau$ associated to a tagged triangulation $\tau$ of a surface $\surf$ with non-empty boundary (and any number of punctures). This is done by passing to the ideal triangulation $\tau^\circ$ obtained by `deletion of notches', 
then reading the potential $S(\tau^\circ)$ defined according to \cite{Lqps}, and then going back to $\tau$ by means of the function $\epsilon:\punct\to\{-1,1\}$ that takes the value 1 precisely at the punctures at which the signature of $\tau$ takes non-negative value ($\punct$ denotes the puncture set of $\surf$). In particular, the potentials obtained for ideal triangulations (which are the tagged triangulations with non-negative signature) are precisely the ones defined in \cite{Lqps}.

Once the QPs $\qstau$ are defined, we prove Theorem \ref{thm:flip<->mutation-tagged-triangs}, the first main result of this paper, which says that any two tagged triangulations $\surf$ are connected by a sequence of flips each of which is compatible with the corresponding QP-mutation. Theorem 30 of \cite{Lqps} (stated below as Theorem \ref{thm:flip<->mutation-ideal-triangs}), which says that ideal triangulations related by a flip give rise to QPs related by QP-mutation, plays an essential role in the proof of this result. From a combination of Theorem \ref{thm:flip<->mutation-tagged-triangs} with Amiot's categorification and the fact, proved by Fomin-Shapiro-Thurston, that the exchange graph of tagged triangulations coincides with the exchange graph of any of the cluster algebras associated to $\surf$, we deduce Corollary \ref{coro:flip<->mutation-jacobian}: if $\tau$ and $\sigma$ are tagged triangulations related by the flip of a tagged arc $\arc$, then the Jacobian algebra of $\qssigma$ is isomorphic to the Jacobian algebra of $\mu_{\arc}\qstau$. Theorem \ref{thm:flip<->mutation-tagged-triangs} and Corollary \ref{coro:flip<->mutation-jacobian} also mean that the QPs $\qstau$ calculate the endomorphism algebras of all the cluster-tilting objects of the generalized cluster category $\mathcal{C}_{\surf}$ that correspond to tagged triangulations.

By definition, $\stau$ is always a finite potential, that is, a finite linear combination of cycles of $\qtau$. It thus makes sense to wonder about the relation between $R\langle \qtau\rangle/J_0(\stau)$ and the Jacobian algebra $\mathcal{P}\qstau$, where $J_0(\stau)$ is the ideal the cyclic derivatives of $\stau$ generate in the non-completed path algebra $R\langle \qtau\rangle$. In Section \ref{sec:admissibility} we prove our second main result, namely, that for any tagged triangulation $\tau$ of a surface with non-empty boundary, the algebra homomorphism $R\langle \qtau\rangle/J_0(\stau)\rightarrow \mathcal{P}\qstau$ induced by the inclusion $R\langle\qtau\rangle\hookrightarrow R\langle\langle\qtau\rangle\rangle$ is an isomorphism. Since $\mathcal{P}\qstau$ is finite-dimensional, this means, in particular, that $J_0(\stau)$ is an admissible ideal of $R\langle\qtau\rangle$ and that one can work with the Jacobian algebra without having to take completions. 

The results from Section \ref{sec:admissibility} allow an application of Derksen-Weyman-Zelevinsky's homological interpretation of the $E$-invariant to obtain results about cluster monomials in the cluster algebras associated to $\surf$. This application is described in Section \ref{sec:monomials} and follows the techniques introduced by the first author in \cite{GCI}. For a cluster algebra $\mathcal{A}$ associated to $\surf$ over an arbitrary coefficient system, we prove that if $\x$ and $\x'$ are two different clusters, then any monomial in $\x'$ in which at least one element from $\x'\setminus\x$ appears with positive exponent is a $\Z\semifield$-linear combination of proper Laurent monomials in $\x$ (that is, Laurent monomials with non-trivial denominator). Then we show that in any cluster algebra satisfying this \emph{proper Laurent monomial property}, cluster monomials are linearly independent over the group ring $\Z\semifield$. We also show that in such a cluster algebra $\myAA$, if a positive element (that is, an element whose Laurent expansion with respect to each cluster has coefficients in $\Z_{\geq 0}\semifield$) belongs to the $\Z\semifield$-submodule of $\mathcal{A}$ generated by all cluster monomials, then it can be written as a $\Z_{\geq 0}\semifield$-linear combination of cluster monomials. As an application of the latter result, in Section \ref{sec:atomic-bases} we show that in coefficient-free skew-symmetric cluster algebras of types $\mathbb{A}$, $\mathbb{D}$ and $\mathbb{E}$, cluster monomials form an atomic basis (that is, a $\Z$-basis $\mathcal{B}$ of $\myAA$ such that the set of positive elements of $\myAA$ is precisely the set of non-negative $\Z$-linear combinations of elements of $\mathcal{B}$)


It is worth mentioning that the results from Section \ref{sec:monomials} are valid over arbitrary coefficient systems. However, for simplicity reasons, we have limited ourselves to work over coefficient systems of geometric type.
Throughout the paper, $\field$ will always denote a field. In Sections \ref{sec:monomials} and \ref{sec:atomic-bases} we will assume that $K=\C$.

\section{Algebraic and combinatorial background}\label{sec:background}

\subsection{Quiver mutations}\label{subsec:quiver-mutations}

In this subsection we recall the operation of quiver mutation, fundamental in Fomin-Zelevinsky's definition of (skew-symmetric) cluster algebras.
Recall that a \emph{quiver} is a finite directed graph, that is, a quadruple $Q=(Q_0,Q_1,h,t)$, where $Q_0$ is the (finite) set of
\emph{vertices} of $Q$, $Q_1$ is the (finite) set of \emph{arrows}, and $h:Q_1\rightarrow Q_0$ and $t:Q_1\rightarrow Q_0$ are the \emph{head}
and \emph{tail} functions. Recall also the common notation $a:i\rightarrow j$ to indicate that $a$ is an arrow of $Q$ with $t(a)=i$, $h(a)=j$.
We will always deal only with loop-free quivers, that is, with quivers that have no arrow $a$ with $t(a)=h(a)$.

A \emph{path of length} $d>0$ in $Q$ is a sequence $a_1a_2\ldots a_d$ of arrows with $t(a_j)=h(a_{j+1})$ for $j=1,\ldots,d-1$. A path
$a_1a_2\ldots a_d$ of length $d>0$ is a $d$\emph{-cycle} if $h(a_1)=t(a_d)$. A quiver is \emph{2-acyclic} if it has no 2-cycles.

Paths are composed as functions, that is, if $a=a_1\ldots a_d$ and $b=b_1\ldots b_{d'}$ are paths with $h(b)=t(a)$, then the product $ab$ is defined as the path $a_1,\ldots a_db_1\ldots b_{d'}$, which starts at $t(b_{d'})$ and ends at $h(a_1)$. See Figure \ref{prodofpaths}.

 \begin{figure}[!h]
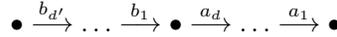

                \caption{Paths are composed as functions}\label{prodofpaths}
                \centering
$$
\bullet\overset{b_{d'}}{\longrightarrow}\ldots\overset{b_1}{\longrightarrow}\bullet\overset{a_d}{\longrightarrow}\ldots\overset{a_1} {\longrightarrow}\bullet
$$
 \end{figure}

For $i\in Q_0$, an $i$\emph{-hook} in $Q$ is any path $ab$ of length 2 such that $a,b\in Q_1$ are arrows with $t(a)=i=h(b)$.

\begin{defi}\label{def:threesteps} Given a quiver $Q$ and a vertex $i\in Q_0$ such that $Q$ has no $2$-cycles incident at $i$, we define the
\emph{mutation} of $Q$ in direction $i$ as the quiver $\muti(Q)$ with vertex set $Q_0$ that results after applying the following three-step
procedure to $Q$:
\begin{itemize}
\item[(Step 1)] For each $i$-hook $ab$ introduce an arrow $[ab]:t(b)\rightarrow h(a)$.
\item[(Step 2)] Replace each arrow $a:i\rightarrow h(a)$ of $Q$ with an arrow $a^*:h(a)\rightarrow i$, and each arrow $b:t(b)\rightarrow i$
of $Q$ with an arrow $b^*:i\rightarrow t(b)$.
\item[(Step 3)] Choose a maximal collection of disjoint 2-cycles and remove them.
\end{itemize}
We call the quiver obtained after the $1^{st}$ and $2^{nd}$ steps the \emph{premutation} $\premuti(Q)$.
\end{defi}
%

\subsection{Cluster algebras}

In this subsection we recall the definition of skew-symmetric cluster algebras of geometric type. Our main references are \cite{FZ1}, \cite{FZ2}, \cite{BFZ3} and \cite{FZ4}.

Let $r$ and $n$ be non-negative integers, with $n\geq 1$. Let $\semifield=\Trop(x_{n+1},\ldots,x_{n+r})$ be the \emph{tropical semifield} on $r$ generators. By definition, $\semifield$ is the free abelian group in $r$ different symbols $x_{n+1},\ldots,x_{n+r}$, with its group operation written multiplicatively, and has the \emph{auxiliary addition} $\oplus$ defined by
$$
\left(x_{n+1}^{a_{n+1}}\ldots x_{n+r}^{a_{n+r}}\right)\oplus \left(x_{n+1}^{b_{n+1}}\ldots x_r^{b_{n+r}}\right)=x_{n+1}^{\min(a_{n+1},b_{n+1})}\ldots x_{n+r}^{\min(a_{n+r},b_{n+r})}.
$$
Thus, the elements of $\semifield$ are precisely the Laurent monomials in the symbols $x_{n+1},\ldots,x_{n+r}$, and the group ring $\Z\semifield$ is the ring of Laurent polynomials in $x_{n+1},\ldots,x_{n+r}$ with integer coefficients. (We warn the reader that the addition of $\Z\semifield$ has absolutely nothing to do with the auxiliary addition $\oplus$ of $\semifield$).

Fix a field $\myFF$ isomorphic to the field of fractions of the ring of polynomials in $n$ algebraically independent variables with coefficients in $\Z\semifield$. A (labeled) \emph{seed} in $\myFF$ is a pair $(\widetilde{B},\x)$, where
\begin{itemize}\item $\x=(x_1,\ldots,x_n)$ is a tuple of $n$ elements of $\myFF$ that are algebraically independent over $\Q\semifield$ and such that $\myFF=\Q\semifield(x_1,\ldots,x_n)$ (such tuples are often called \emph{free generating sets of $\myFF$});
\item $\widetilde{B}$ is an $(n+r)\times n$ integer matrix, whose first $n$ rows form a skew-symmetric matrix $B$.
\end{itemize}
The matrix $B$ (resp. $\widetilde{B}$) receives the name of \emph{exchange matrix} (resp. \emph{extended exchange matrix}) of the seed $(\widetilde{B},\x)$, whereas the tuple $\x$ is called the (ordered) \emph{cluster} of the seed.

\begin{defi}\label{def:seed-mutation} Let $(\widetilde{B},\x)$ be a seed. For $i\in[1,n]=\{1,\ldots,n\}$, the
\emph{mutation of $(\widetilde{B},\x)$ with respect to $i$}, denoted by $\mu_i(\widetilde{B},\x)$, is the pair $(\widetilde{B}',\x')$, where
\begin{itemize}\item $\widetilde{B}'$ is the $(n+r)\times n$ integer matrix whose entries are defined by
$$
b'_{kj}=\begin{cases}
  -b_{kj} & \text{if} \ k=i \ \text{or} \ j=i,\\
  b_{kj}+\frac{b_{ki}|b_{ij}|+|b_{ki}|b_{ij}}{2} & \text{if} \ k\neq i\neq j;
\end{cases}
$$
\item $\x'=(x'_1,\ldots,x'_n)$ is the $n$-tuple of elements of $\myFF$ given by
\begin{equation}\label{eq:cluster-mutation}
x_k'=\begin{cases}
 x_k & \text{if} \ k\neq i,\\
 \frac{\prod_{j=1}^{n+r}x_j^{[b_{ji}]_+}+\prod_{j=1}^{n+r}x_j^{[-b_{ji}]_+}}{x_i} & \text{if} \ k=i,
\end{cases}
\end{equation}
where $[b]_+=\max(0,b)$ for any real number $b$.
\end{itemize}
\end{defi}

It is easy to check that $\mu_i$ is an involution of the set of all seeds of $\myFF$. That is, if $(\widetilde{B},\x)$ is a seed in $\myFF$, then $\mu_i(\widetilde{B},\x)$ is a seed in $\myFF$ as well, and $\mu_i\mu_i(\widetilde{B},\x)=(\widetilde{B},\x)$. Two seeds are \emph{mutation equivalent} if one can be obtained from the other by applying a finite sequence of seed mutations.

Let $\TT$ be an $n$-regular tree, with each of its edges labeled by a number from the set $[1,n]$ in such a way that different edges incident to the same vertex have different labels.
A \emph{cluster pattern} assigns to each vertex $t$ of $\TT$ a seed $(\widetilde{B}_t,\x_t)$, in such a way that whenever two vertices $t,t'$, of $\TT$ are connected by an edge labeled with the number $i$, their corresponding seeds are related by seed mutation with respect to $i$. It is clear that if we fix a vertex $t_0$ of $\TT$, any cluster pattern is uniquely determined by the choice of an \emph{initial seed} $(\widetilde{B}_{t_0},\x_{t_0})$. By definition, the (skew--symmetric) \emph{cluster algebra} $\myAA(\widetilde{B}_{t_0},\x_{t_0})$ associated to the seed $(\widetilde{B}_{t_0},\x_{t_0})$ is the $\Z\semifield$-subalgebra of $\myFF$ generated by union of all clusters of the seeds mutation equivalent to $(\widetilde{B}_{t_0},\x_{t_0})$. Note that, up to field automorphisms of $\myFF$, the cluster algebra $\myAA(\widetilde{B}_{t_0},\x_{t_0})$ depends only on the \emph{initial extended exchange matrix} $\widetilde{B}_{t_0}$. Hence it is customary to write $\myAA(\widetilde{B}_{t_0})=\myAA(\widetilde{B}_{t_0},\x_{t_0})$.

Because of \eqref{eq:cluster-mutation}, for every vertex $t$ of $\TT$, all cluster variables in $\myAA(\widetilde{B}_{t_0})$ can be expressed as rational functions in $\x_t$ with coefficients in $\Z\semifield$. One of the fundamental theorems of cluster algebra theory, the famous \emph{Laurent phenomenon} of Fomin-Zelevinsky, asserts that these rational functions are actually Laurent polynomials in $\x_t$ with coefficients in $\Z\semifield$. 
Fomin-Zelevinsky have conjectured that all integers appearing in these Laurent expansions are non-negative.

The language of quivers turns out to be extremely useful to obtain information about cluster algebras. Each skew-symmetric matrix $B$ gives rise to a quiver $Q=Q(B)$ as follows: the vertex set of $Q$ is $Q_0=[1,n]$, and for each pair of vertices $i,j$, $Q$ has $b_{ij}$ arrows from $j$ to $i$ provided $b_{ij}\geq 0$. The quiver counterpart of the mutation rule for matrices stated in Definition \ref{def:seed-mutation} is precisely the quiver mutation described in Definition \ref{def:threesteps}.

\subsection{Quivers with potentials}

Here we give the background on quivers with potentials and their mutations we shall use in the remaining sections. Our main reference for this subsection is \cite{DWZ1}.  A survey of the topics treated in \cite{DWZ1} can be found in \cite{ZOberwolfach2007}.

Given a quiver $Q$, we denote by $R$ the $\field$-vector space with basis $\{e_i\suchthat i\in Q_0\}$. If we define $e_ie_j=\delta_{ij}e_i$,
then $R$ becomes naturally a commutative semisimple $\field$-algebra, which we call the \emph{vertex span} of $Q$; each $e_i$ is called the
\emph{path of length zero} at $i$. We define the \emph{arrow span} of $Q$ as the $\field$-vector space $A$ with basis the set of arrows $Q_1$.
Note that $A$ is an $R$-bimodule if we define $e_ia=\delta_{i,h(a)}a$ and $ae_j=a\delta_{t(a),j}$ for $i\in Q_0$ and $a\in Q_1$. For $d\geq 0$
we denote by $A^d$ the $\field$-vector space with basis all the paths of length $d$ in $Q$; this space has a natural $R$-bimodule structure as well. Notice that $A^0=R$ and $A^1=A$.

\begin{defi} The \emph{complete path algebra} of $Q$ is the $\field$-vector space consisting of all possibly infinite linear combinations of paths in $Q$,
that is,
\begin{equation}
\completeRQ=\underset{d=0}{\overset{\infty}{\prod}}A^d;
\end{equation}
with multiplication induced by concatenation of paths.
\end{defi}

Note that $\completeRQ$ is a $\field$-algebra and an $R$-bimodule, and has the usual \emph{path
algebra}
\begin{equation}
\usualRQ=\underset{d=0}{\overset{\infty}{\bigoplus}}A^d
\end{equation}
as $\field$-subalgebra and sub-$R$-bimodule. Moreover, $\usualRQ$ is dense in $\completeRQ$ under the $\idealM$-adic topology, whose fundamental system
of open neighborhoods around $0$ is given by the powers of $\idealM=\idealM(Q)=\prod_{d\geq 1}A^d$, which is the ideal of $\completeRQ$ generated by the arrows. A crucial property of this
topology is the following:
\begin{equation}\label{eq:convergence-in-R<<Q>>}
\text{a sequence $(x_n)_{n\in\mathbb{N}}$ of elements of $\completeRQ$ converges if and only if for every $d\geq 0$},
\end{equation}
\begin{center}
the sequence $(x_n^{(d)})_{n\in\mathbb{N}}$ stabilizes as $n\rightarrow\infty$, in which case
$\underset{n\rightarrow\infty}{\lim}x_n=\underset{d\geq 0}{\sum}\underset{n\rightarrow\infty}{\lim}x_n^{(d)}$,
\end{center}
where $x_n^{(d)}$ denotes the degree-$d$ component of $x_n$.

Even though the action of $R$ on $\completeRQ$ (and $\usualRQ$) is not central, it is compatible with the multiplication of $\completeRQ$ in the sense that if
$a$ and $b$ are paths in $Q$, then $e_{h(a)}ab=ae_{t(a)}b=abe_{t(b)}$. Therefore we will say that $\completeRQ$ (and $\usualRQ$) are
$R$-algebras. Accordingly, any $K$-algebra homomorphism $\varphi$ between (complete) path algebras will be called an $R$-algebra homomorphism
if the underlying quivers have the same set of vertices and $\varphi(r)=r$ for every $r\in R$. It is easy to see that every $R$-algebra
homomorphism between complete path algebras is continuous. The following is an extremely useful criterion to decide if a given linear map
$\varphi:\completeRQ\rightarrow R\langle\langle Q'\rangle\rangle$ between complete path algebras (on the same set of vertices) is an $R$-algebra homomorphism or an $R$-algebra
isomorphism (for a proof, see \cite[Proposition 2.4]{DWZ1}):
\begin{align}\label{eq:bimodule-homomorphisms}
& \text{Every pair $(\varphi^{(1)},\varphi^{(2)})$ of $R$-bimodule homomorphisms $\varphi^{(1)}:A\rightarrow A'$, $\varphi^{(2)}:A\rightarrow\idealM(Q')^2$,} \\
\nonumber
& \text{extends uniquely to a continuous $R$-algebra homomorphism $\varphi:\completeRQ\rightarrow R\langle\langle Q'\rangle\rangle$} \\
\nonumber
& \text{such that $\varphi|_A=(\varphi^{(1)},\varphi^{(2)})$.
Furthermore, $\varphi$ is $R$-algebra isomorphism if and only if}\\
\nonumber
 & \text{$\varphi^{(1)}$ is an $R$-bimodule isomorphism.}
\end{align}

There are many definitions surrounding Derksen-Weyman-Zelevinsky's mutation theory of quivers with potentials. In order to be as concise as possible, but still self-contained, we present the most important ones (for our purposes) at once.

\begin{defi}\label{def:QP-stuff}\begin{itemize}\item A \emph{potential} on $Q$ (or $A$) is any element of $\completeRQ$ all of whose terms are cyclic paths of positive length. The set of all potentials on
$Q$ is denoted by $\pot$, it is a closed vector subspace of $\completeRQ$.
\item Two potentials $S,S'\in\pot$ are \emph{cyclically equivalent} if $S-S'$
lies in the closure of the vector subspace of $\completeRQ$ spanned by all the elements of the form $a_1\ldots a_d-a_2\ldots a_da_1$ with $a_1\ldots
a_d$ a cyclic path of positive length.
\item A \emph{quiver with potential} is a pair $(Q,S)$ (or $(A,S)$), where $S$ is a potential on $Q$ such that no two different cyclic paths
appearing in the expression of $S$ are cyclically equivalent. Following \cite{DWZ1}, we will use the shorthand \emph{QP} to abbreviate ``quiver
with potential".
\item The \emph{direct sum} of two QPs $(A,S)$ and $(A',S')$ on the same set of vertices is the QP $(A,S)\oplus(A',S')=(A\oplus A',S+S')$. (The $R$-bimodule $A\oplus A'$ is the arrow span of the quiver whose vertex set is $Q_0$ and whose arrow set is $Q_1\sqcup Q_1'$.)
\item If $(Q,S)$ and $(Q',S')$ are QPs on the same set of vertices, we say that $(Q,S)$ is \emph{right-equivalent} to $(Q',S')$ if there exists a
\emph{right-equivalence} between them, that is, an $R$-algebra isomorphism $\varphi:\completeRQ\rightarrow R\langle\langle Q'\rangle\rangle$ such that $\varphi(S)$ is cyclically
equivalent to $S'$.
\item For each arrow $a\in Q_1$ and each cyclic path $a_1\ldots a_d$ in $Q$ we define the \emph{cyclic derivative}
\begin{equation}
\partial_a(a_1\ldots a_d)=\underset{i=1}{\overset{d}{\sum}}\delta_{a,a_i}a_{i+1}\ldots a_da_1\ldots a_{i-1},
\end{equation}
(where $\delta_{a,a_i}$ is the \emph{Kronecker delta}) and extend $\partial_a$ by linearity and continuity to obtain a map $\partial_a:\pot\rightarrow\completeRQ$.
\item The \emph{Jacobian ideal} $J(S)$ is the closure of the two-sided ideal of $\completeRQ$ generated by $\{\partial_a(S)\suchthat a\in Q_1\}$, and the
\emph{Jacobian algebra} $\jacobqs$ is the quotient algebra $\completeRQ/J(S)$.
\item A QP $(Q,S)$ is \emph{trivial} if $S\in A^2$ and $\{\partial_a(S)\suchthat a\in Q_1\}$ spans $A$.
\item A QP $(Q,S)$ is \emph{reduced} if the degree-$2$ component of $S$ is $0$, that is,
if the expression of $S$ involves no $2$-cycles.
\item We say that a quiver $Q$
(or its arrow span, or any QP having it as underlying quiver) is \emph{2-acyclic} if it has no $2$-cycles.
\end{itemize}
\end{defi}

\begin{prop}\cite[Propositions 3.3 and 3.7]{DWZ1} \begin{enumerate}\item If $S,S'\in\pot$ are cyclically equivalent, then
$\partial_a(S)=\partial_a(S')$ for all $a\in Q_1$.
\item Jacobian ideals and Jacobian algebras are invariant under
right-equivalences. That is, if $\varphi:\completeRQ\rightarrow R\langle\langle Q'\rangle\rangle$ is a right-equivalence between $(Q,S)$ and $(Q',S')$, then $\varphi$
sends $J(S)$ onto $J(S')$ and therefore induces an isomorphism $\jacobqs\rightarrow\mathcal{P}(Q',S')$.
\end{enumerate}
\end{prop}

One of the main technical results of \cite{DWZ1} is the \emph{Splitting Theorem}, which we now state.

\begin{thm}\cite[Theorem 4.6]{DWZ1}\label{thm:splittingthm} For every QP $(Q,S)$ there exist a trivial QP
$\qstrivial$ and a reduced QP $\qsreduced$ such that $(Q,S)$ is right-equivalent to the direct sum $\qstrivial\oplus\qsreduced$. Furthermore,
the right-equivalence class of each of the QPs $\qstrivial$ and $\qsreduced$ is determined by the right-equivalence class of $(Q,S)$.
\end{thm}

In the situation of Theorem \ref{thm:splittingthm}, the QPs $\qsreduced$ and $\qstrivial$ are called, respectively, the \emph{reduced part} and the \emph{trivial
part} of $(Q,S)$; this terminology is well defined up to right-equivalence.

We now turn to the definition of mutation of a QP. Let $(Q,S)$ be a QP on the vertex set $Q_0$ and let $i\in Q_0$. Assume that $Q$ has no
2-cycles incident to $i$. Thus, if necessary, we replace $S$ with a cyclically equivalent potential so that we can assume that every cyclic
path appearing in the expression of $S$ does not begin at $i$. This allows us to define $[S]$ as the potential on $\premuti(Q)$ obtained from
$S$ by replacing each $i$-hook $ab$ with the arrow $[ab]$ (see the line preceding Definition \ref{def:threesteps}). Also, we define $\Delta_i(Q)=\sum b^*a^*[ab]$,
where the sum runs over all $i$-hooks $ab$ of $Q$.

\begin{defi}\label{def:QP-mutation} Under the assumptions and notation just stated, we define the \emph{premutation} of $(Q,S)$ in direction $i$ as the QP
$\premuti(Q,S)=(\premuti(Q),\premuti(S))$ (see Definition \ref{def:threesteps}), where
$\premuti(S)=[S]+\Delta_i(Q)$. The \emph{mutation} $\muti(Q,S)$ of $(Q,S)$ in direction $i$ is then defined as the reduced part of
$\premuti(Q,S)$.
\end{defi}

%

\begin{defi}\begin{itemize}\item A QP $(Q,S)$ is \emph{non-degenerate} if it is 2-acyclic and the QP obtained after any possible sequence of QP-mutations is 2-acyclic.
\item A QP $(Q,S)$ is \emph{rigid} if every cycle in $Q$ is cyclically equivalent to an element of the Jacobian ideal $J(S)$.
\end{itemize}
\end{defi}

The next theorem summarizes the main results of \cite{DWZ1} concerning QP-mutations. The reader can find the statements and their respective proofs in Theorem 5.2 and Corollary 5.4, Theorem 5.7, Corollary 7.4, Corollary 6.11, and Corollary 6.6, of \cite{DWZ1}.

\begin{thm}\label{thm:propertiesofQP-mutations}\begin{enumerate}\item Premutations and mutations are well defined up to right-equivalence.
\item Mutations are involutive up to right-equivalence.
\item If the base field $\field$ is uncountable, then every 2-acyclic quiver admits a non-degenerate QP.
\item The class of reduced rigid QPs is closed under QP-mutation. Consequently, every rigid reduced QP is non-degenerate.
\item Finite-dimensionality of Jacobian algebras is invariant under QP-mutations.
\end{enumerate}
\end{thm}

\subsection{QP-representations}

In this subsection we describe how the notions of right-equivalence and QP-mutation extend to the level of representations. As in the previous subsection, our main reference is \cite{DWZ1}.

\begin{defi} Let $(Q,S)$ be any QP. A \emph{decorated} $(Q,S)$\emph{-representation}, or \emph{QP-representation}, is a quadruple $\mathcal{M}=(Q,S,M,V)$, where $M$ is a finite-dimensional left $\jacobqs$-module and $V$ is a finite-dimensional left $R$-module.
\end{defi}

By setting $M_i=e_iM$ for each $i\in Q_0$, and $a_M:M_{t(a)}\rightarrow M_{h(a)}$ as the multiplication by $a\in Q_1$ given by the $\completeRQ$-module structure of $M$, we easily see that each $\mathcal{P}(Q,S)$-module induces a representation of the quiver $Q$. Furthermore, since every finite-dimensional $\completeRQ$-module is nilpotent (that is, annihilated by some power of $\idealM$) any QP-representation is prescribed by the following data:
\begin{enumerate}\item A tuple $(M_i)_{i\in Q_0}$ of finite-dimensional $K$-vector spaces;
\item a family $(a_M:M_{t(a)}\rightarrow M_{h(a)})_{a\in Q_0}$ of $K$-linear transformations annihilated by $\{\partial_a(S)\suchthat a\in Q_1\}$, for which there exists an integer $r\geq 1$ with the property that the composition ${a_1}_{M}\ldots {a_r}_{M}$ is identically zero for every $r$-path $a_1\ldots a_r$ in $Q$.
\item a tuple $(V_i)_{i\in Q_0}$ of finite-dimensional $\field$-vector spaces (without any specification of linear maps between them).
\end{enumerate}

\begin{remark}
In the literature, the linear map $a_M:M_{t(a)}\rightarrow M_{h(a)}$ induced by left multiplication by $a$ is more commonly denoted by $M_a$. We will use both of these notations indistinctly.
\end{remark}

\begin{defi} Let $(Q,S)$ and $(Q',S')$ be QPs on the same set of vertices, and let $\mathcal{M}=(Q,S,M,V)$ and $\mathcal{M}'=(Q',S',M',V')$ be decorated representations. A triple $\Phi=(\varphi,\psi,\eta)$ is called a \emph{right-equivalence} between $\mathcal{M}$ and $\mathcal{M}'$ if the following three conditions are satisfied:
\begin{itemize}\item $\varphi:\completeRQ\rightarrow R\langle\langle Q'\rangle\rangle$ is a right-equivalence of QPs between $(Q,S)$ and $(Q',S')$;
\item $\psi:M\rightarrow M'$ is a vector space isomorphism such that $\psi\circ u_M=\varphi(u)_{M'}\circ\psi$ for all $u\in\completeRQ$;
\item $\eta:V\rightarrow V'$ is an $R$-module isomorphism.
\end{itemize}
\end{defi}


Recall that every QP is right-equivalent to the direct sum of its reduced and trivial parts, which are determined up to right-equivalence (Theorem \ref{thm:splittingthm}). These facts have representation-theoretic extensions, which we now describe. Let $(Q,S)$ be
any QP, and let $\varphi:R\langle\langle Q_{\ored}\oplus C\rangle\rangle\rightarrow\completeRQ$ be a right equivalence between
$(Q_{\ored},S_{\ored})\oplus(C,T)$ and $(Q,S)$, where $(Q_{\ored},S_{\ored})$ is a reduced QP and $(C,T)$ is a trivial QP. Let
$\mathcal{M}=(Q,S,M,V)$ be a decorated representation, and set $M^\varphi=M$ as $K$-vector space. Define an action of $R\langle\langle
Q_{\operatorname{red}}\rangle\rangle$ on $M^\varphi$ by setting $u_{M^\varphi}=\varphi(u)_M$ for $u\in R\langle\langle
Q_{\operatorname{red}}\rangle\rangle$.

\begin{prop}\cite[Propositions 4.5 and 10.5]{DWZ1}\label{prop:reddetermineduptorequiv} With the action of $R\langle\langle Q_{\operatorname{red}}\rangle\rangle$ on $M^\varphi$ just defined, the quadruple $(Q_{\ored},S_{\ored},M^\varphi,V)$ becomes a QP-representation. Moreover, the right-equivalence class of $(Q_{\ored},S_{\ored},M^\varphi,V)$ is determined by the right-equivalence class of $\mathcal{M}$.
\end{prop}

The (right-equivalence class of the) QP-representation $\mathcal{M}_{\ored}=(Q_{\ored},S_{\ored},M^\varphi,V)$ is the \emph{reduced part} of $\mathcal{M}$.


We now turn to the representation-theoretic analogue of the notion of QP-mutation. Let $(Q,S)$ be a QP. Fix a vertex $i\in Q_0$, and suppose that $Q$ does not have 2-cycles incident to $i$. Denote by $a_1,\ldots,a_s$ (resp. $b_1,\ldots,b_t$) the arrows ending at $i$ (resp. starting at $i$). Take a QP-representation $\mathcal{M}=(Q,S,M,V)$ and set
$$
M_{\oin}=M_{\oin}(i)=\underset{k=1}{\overset{s}{\bigoplus}}M_{t(a_k)}, \ \ M_{\oout}=M_{\oout}(i)=\underset{l=1}{\overset{t}{\bigoplus}}M_{h(b_l)}.
$$
Multiplication by the arrows $a_1,\ldots,a_s$ and $b_1,\ldots,b_t$ induces $K$-linear maps
$$
\al=\al_i=[a_1 \ \ldots \ a_s]:M_{\oin}\rightarrow M_i, \ \ \be=\be_i=\left[\begin{array}{c}b_1 \\ \vdots \\ b_t\end{array} \right]:M_i\rightarrow M_{\oout}.
$$

For each $k\in[1,s]$ and each $l\in[1,t]$ let $\ga_{k,l}:M_{h(b_l)}\rightarrow M_{t(a_k)}$ be the linear map given by multiplication by the the element $\partial_{[b_la_k]}([S])$, and arrange these maps into a matrix to obtain a linear map $\ga=\ga_i:M_{\oout}\rightarrow M_{\oin}$ (remember that $[S]$ is obtained from $S$ by replacing each $i$-hook $ab$ with the arrow $[ab]$). Since $M$ is a $\jacobqs$-module, we have $\al\ga=0$ and $\ga\be=0$.

Define vector spaces $\overline{M}_j=M_j$ and $\overline{V}_j=V_j$ for $j\in Q_0$, $j\neq i$, and
$$
\overline{M}_i=\frac{\ker\ga}{\image\be}\oplus\image\ga\oplus\frac{\ker\al}{\image\ga}\oplus V_i, \ \ \ \ \overline{V}_i=\frac{\ker\be}{\ker\be\cap\image\al}.
$$

We define an action of the arrows of $\widetilde{\mu}_i(Q)$ on $\overline{M}=\underset{j\in Q_0}{\bigoplus}\overline{M}_j$ as follows. If $c$ is an arrow of $Q$ not incident to $i$, we define $c_{\overline{M}}=c_M$, and for each $k\in[1,s]$ and each $l\in[1,t]$ we set $[b_la_k]_{\overline{M}}=(b_la_k)_M={b_l}_M{a_k}_M$. To define the action of the remaining arrows, choose a linear map $\rh=\rh_i:M_{\oout}\rightarrow\ker\ga$ such that the composition $\ker\ga\hookrightarrow M_{\oout}\overset{\rh}{\rightarrow}\ker\ga$ is the identity (where $\ker\ga\hookrightarrow M_{\oout}$ is the inclusion) and a linear map $\si=\si_i:\frac{\ker\al}{\image\ga}\rightarrow\ker\al$ such that the composition $\frac{\ker\al}{\image\ga}\overset{\si}{\rightarrow}\ker\al\twoheadrightarrow\frac{\ker\al}{\image\ga}$ is the identity (where $\ker\al\twoheadrightarrow\frac{\ker\al}{\image\ga}$ is the canonical projection). Then set
$$
[b_1^* \ \ldots \ b_t^*]=\overline{\al}=\left[\begin{array}{c}-\pipi\rh \\ -\ga \\ 0 \\ 0\end{array}\right]:M_{\oout}\rightarrow\overline{M}_i, \ \ \ \  \left[\begin{array}{c}a_1^* \\ \vdots \\ a_s^*\end{array} \right]=\overline{\be}=[0 \ \io \ \io\si \ 0]:\overline{M}_i\rightarrow M_{\oin},
$$
where $\pipi:\ker\ga\rightarrow\frac{\ker\ga}{\image\be}$ is the canonical projection and $\io:\ker\al\rightarrow M_{\oin}$ is the inclusion.

Since $\idealM^{r}M=0$ for some sufficiently large $r$, this action of the arrows of $\widetilde{\mu}_i(Q)$ on $\overline{M}$ extends uniquely to an action of $R\langle\langle\widetilde{\mu}_i(Q)\rangle\rangle$ under which $\overline{M}$ is an $R\langle\langle\widetilde{\mu}_i(Q)\rangle\rangle$-module.

\begin{remark} Note that the choice of the linear maps $\rh$ and $\si$ is not canonical. However, different choices lead to isomorphic $R\langle\langle\widetilde{\mu}_i(Q)\rangle\rangle$-module structures on $\overline{M}$.
\end{remark}

\begin{defi} With the above action of $R\langle\langle\widetilde{\mu}_i(Q)\rangle\rangle$ on $\overline{M}$ and the obvious action of $R$ on $\overline{V}=\bigoplus_{j\in Q_0}\overline{V}_j$, the quadruple $(\widetilde{\mu}_i(Q),\widetilde{\mu}_i(S),\overline{M},\overline{V})$ is called the \emph{premutation}  of $\mathcal{M}=(Q,S,M,V)$ in direction $j$, and denoted $\widetilde{\mu}_i(\mathcal{M})$. The \emph{mutation} of $\mathcal{M}$ in direction $i$, denoted by $\mu_i(\mathcal{M})$, is the reduced part of $\widetilde{\mu}_i(\mathcal{M})$.
\end{defi}

The following are important properties of mutations of QP-representations. Proofs can be found in Propositions 10.7 and 10.10, Corollary 10.12, and Theorem 10.13 of \cite{DWZ1}.

\begin{thm}\begin{enumerate}\item Premutations and mutations are well defined up to right-equivalence.
\item Mutations of QP-representations are involutive up to right-equivalence.
\end{enumerate}
\end{thm}

To close the current subsection, we recall Derksen-Weyman-Zelevinsky's definition of the $\g$-vector, the $F$-polynomial and the $E$-invariant of a QP-representation (cf. \cite{DWZ2}). Let $(Q,S)$ be a non-degenerate QP defined over the field $\mathbb{C}$ of complex numbers, and $\mathcal{M}=(M,V)$ be a decorated $(Q,S)$-representation. The $\g$-vector and $F$-polynomial of $\mathcal{M}$ are defined as follows. For $i\in Q_0$, the $i^{\operatorname{th}}$ entry of the $\g$-vector $\g_{\mathcal{M}}$ of $\mathcal{M}$ is
\begin{equation}
g_i^{\mathcal{M}}=\dim\ker\ga_{i}-\dim M_{i}+\dim V_{i}.
\end{equation}
The $F$-polynomial $F_{\mathcal{M}}$ of $\mathcal{M}$ is
\begin{equation}
F_{\mathcal{M}}=\sum_{\e\in\mathbb{N}^{Q_0}}\chi(\Gr_{\e}(M))\X^{\e},
\end{equation}
where $\X$ is a set of indeterminates, indexed by $Q_0$, that are algebraically independent over $\mathbb{Q}$, $\Gr_{\e}(M)$ is the \emph{quiver Grassmannian} of subrepresentations of $M$ with dimension vector $\e$, and $\chi$ is the Euler-Poincar\'e characteristic.

Let $\mathcal{N}=(N,W)$ be a second decorated representation of $(Q,S)$. Define the integer 
\begin{equation}\label{Eq:DefEInj}
E^{inj}(\mathcal{M},\mathcal{N})=\operatorname{dim}\operatorname{Hom}_{\mathcal{P}(A,S)}(M,N)+\mathbf{dim}(M)\cdot\mathbf{g}_{\mathcal{N}}
\end{equation}
where $\mathbf{dim}(M)=(\operatorname{dim }M_i)_{i\in Q_0}$ is the dimension vector of the positive part $M$ of $\mathcal{M}$, and $\cdot$ denotes the usual scalar product of vectors. The \emph{$E$-invariant} of $\mathcal{M}$ is then defined to be the integer
$$
E(\mathcal{M})=E^{inj}(\mathcal{M},\mathcal{M}).
$$
According to \cite[theorem~7.1]{DWZ2}, this number is invariant under mutations of QP-representations, that is, $E(\mu_i(\mathcal{M}))=E(\mathcal{M})$ for any decorated representation $\mathcal{M}$. Notice that $E(\mathcal{N})=0$ for every negative QP-representation. Hence $E(\mathcal{M})=0$ for every QP-representation mutation that can be obtained from a negative one by performing a finite sequence of mutations.

If the potential $S$ satisfies an additional ``admissibility condition", the $E$-invariant possesses a remarkable homological interpretation. We will study this in Section \ref{sec:admissibility}.

\subsection{Relation between cluster algebras and QP-representations}

The articles \cite{FZ4} and \cite{DWZ2} are our main references for this subsection, throughout which  $B$ will be an $n\times n$ skew-symmetric integer matrix, and $\widetilde{B}$ will be the $2n\times n$ matrix whose top $n\times n$ submatrix is $B$ and whose bottom $n\times n$ submatrix is the identity matrix. Put a seed $(\widetilde{B},\x)$ as the initial seed of a cluster pattern, being $t_0$ the initial vertex of $\TT$. In \cite{FZ4}, Fomin-Zelevinsky introduce a $\Z^n$-grading for $\Z[x_1^{\pm1},\ldots,x_n^{\pm1},x_{n+1},\ldots,x_{2n}]$ defined by the formulas
$$
\deg(x_l)=\e_l, \ \ \text{and} \ \ \deg(x_{n+l})=-\mathbf{b}_l \ \ \text{for} \ l=1,\ldots,n
$$
where $\e_1,\ldots,\e_n$ are the standard basis (column) vectors in $\Z^n$, and $\mathbf{b}_l=\sum_kb_{kl}\e_k$ is the $l^{\operatorname{th}}$ column of the $n\times n$ top part $B$ of $\widetilde{B}$. Under this $\Z^n$-grading, all cluster variables in the principal-coefficient cluster algebra $\mathcal{A}_{\bullet}(B)=\myAA(\widetilde{B})$ are homogeneous elements (cf. \cite[Proposition 6.1 and Corollary 6.2]{FZ4}). By definition, the \emph{$\g$-vector} $\g_x^{\x}$ of a cluster variable $x\in \mathcal{A}_{\bullet}(B)$ with respect to the initial cluster $\x$ is its multi-degree with respect to the $\Z^n$-grading just defined.

For each vertex $t$ of $\TT$ and each index $l\in[1,n]$, let us denote by $X^{B;t_0}_{k;t}$ the $k^{th}$ cluster variable from the cluster attached to the vertex $t$ by the cluster pattern under current consideration. By \eqref{eq:cluster-mutation}, $X^{B;t_0}_{k;t}$ is a rational function in the initial cluster $\x$ with coefficients in the group semiring $\Z_{\geq 0}\textrm{Trop}(x_{n+1},\ldots,x_{2n})$. The \emph{$F$-polynomial}
$F^{B;t_0}_{k;t}$ is defined to be the result of specializing $x_{1},\ldots,x_{n}$, to $1$ in the rational function $X^{B;t_0}_{k;t}$.

According to \cite[Corollary 6.3]{FZ4}, for $k=1,\ldots,n$ we have the following \emph{separation of additions}.

\begin{thm}\label{thm:FZ-separation-of-variables} Let $\semifield$ be any semifield and let $\mathcal{A}$ be a cluster algebra over the ground ring $\Z\semifield$, contained in the ambient field $\myFF$, with the matrix $B$, the cluster $\x$ and the coefficient tuple $\mathbf{y}$ placed at the initial seed of the corresponding cluster pattern. Let $t$ be a vertex of $\TT$ and $\x'$ be the cluster attached to $t$ by the alluded cluster pattern. Then the $k^{\operatorname{th}}$ cluster variable from $\x'$ has the following expression in terms of the initial cluster $\x$:
\begin{equation}\label{eq:general-separation-formula}
x'_{k}=\frac{F_{k;s}^{B';t}|_{\myFF}(\widehat{y}_1,\ldots,\widehat{y}_n)}{F_{k;s}^{B';t}|_{\semifield}(y_1,\ldots,y_n)}
x_1^{g_{1,k}}\ldots x_n^{g_{n,k}},
\end{equation}
$$
\text{where} \ \ \ \widehat{y}_j=y_j\prod_{i=1}^{n}x_i^{b_{ij}} \ \ \ \text{and} \ \ \ \g_{x'_k}^{\x}=\left[\begin{array}{c}g_{1,k}\\ \vdots\\ g_{n,k}\end{array}\right].
$$
\end{thm}

By the following theorem of Derksen-Weyman-Zelevinsky, $\g$-vectors and $F$-polynomials of QP-representations provide a fundamental representation-theoretic interpretation of $\g$-vectors and $F$-polynomials of cluster variables.

\begin{thm}\cite[Theorem 5.1]{DWZ2}\label{thm:DWZ-gvectors-Fpoynomials} Let $B$ and $\widetilde{B}$ be as in the beginning of the subsection. Attach a seed $(\widetilde{B},\x)$ to the initial vertex $t_0$ of $\TT$ and consider the resulting cluster pattern. For every vertex $t$ and every $k\in[1,n]$ we have
$$
F^{B;t_0}_{k;t}=F_{\mathcal{M}^{B;t_0}_{k;t}} \ \ \ \text{and} \ \ \ \g^{\x}_{X^{B;t_0}_{k;t}}=\g_{\mathcal{M}^{B;t_0}_{k;t}},
$$
where $\mathcal{M}^{B;t_0}_{k;t}$ is a QP-representation defined as follows. Let $t_0\frac{\phantom{xxx}}{i_1}t_1\ldots t_{m-1}\frac{\phantom{xxx}}{i_m}t$ be the unique path on $\TT$ that connects $t$ with $t_0$. Over the field $\mathbb{C}$ of complex numbers, let $S$ be a non-degenerate potential on the quiver $Q=Q(B)$ and $(Q_t,S_t)$ be the QP obtained by applying the QP-mutation sequence $\mu_{i_1},\ldots,\mu_{i_m}$, to $(Q,S)$. Then
$$
\mathcal{M}^{B;t_0}_{k,t}=\mu_{1_1}\mu_{i_2}\ldots\mu_{im}\left(\mathcal{S}^-_k(Q_t,S_t)\right).
$$
\end{thm}

\section{Triangulations of surfaces}\label{sec:triangulations}

For the convenience of the reader, and in order to be as self-contained as possible, we briefly review the material on tagged triangulations of surfaces and their signed adjacency quivers and flips. 
Our main reference for this section is \cite{FST}.

\begin{defi}\label{def:surf-with-marked-points} A \emph{bordered surface with marked points}, or simply a \emph{surface}, is a pair $\surf$, where $\surfnoM$ is a compact connected oriented Riemann surface with (possibly empty) boundary, and $\marked$ is a finite set of points on $\surfnoM$, called \emph{marked points}, such that $\marked$ is non-empty and has at least one point from each connected component of the boundary of $\surfnoM$. The marked points that lie in the interior of $\surfnoM$ are called \emph{punctures}, and the set of punctures of $\surf$ is denoted $\punct$. Throughout the paper we will always assume that:
\begin{itemize}
\item $\surfnoM$ has non-empty boundary;
\item $\surf$ is not an unpunctured monogon, digon or triangle, nor a once-punctured monogon or digon.
\end{itemize}
Here, by a monogon (resp. digon, triangle) we mean a disk with exactly one (resp. two, three) marked point(s) on the boundary.
\end{defi}



\begin{defi}\label{def:ordinaryarcs-compatibility-idealtriangulations} Let $\surf$ be a surface.
\begin{enumerate}\item An \emph{ordinary arc}, or simply an \emph{arc} in $\surf$, is a curve $\arc$ in $\surfnoM$ such that:
\begin{itemize}
\item the endpoints of $\arc$ belong to $\marked$;
\item $\arc$ does not intersect itself, except that its endpoints may coincide;
\item the relative interior of $\arc$ is disjoint from $\marked$ and from the boundary of $\surfnoM$;
\item $\arc$ does not cut out an unpunctured monogon nor an unpunctured digon.
\end{itemize}
\item An arc whose endpoints coincide will be called a \emph{loop}.
\item Two arcs $\arc_1$ and $\arc_2$ are \emph{isotopic} rel $\marked$ if there exists an isotopy $H:I\times\surfnoM\rightarrow\surfnoM$ such that $H(0,x)=x$ for all $x\in\surfnoM$, $H(1,\arc_1)=\arc_2$, and $H(t,m)=m$ for all $t\in I$ and all $m\in\marked$). Arcs will be considered up to isotopy rel $\marked$ and orientation. We denote the set of arcs in $\surf$, considered up to isotopy rel $\marked$ and orientation, by $\arcsinsurf$.
\item Two arcs are \emph{compatible} if there are arcs in their respective isotopy classes whose relative interiors do not intersect.
\item An \emph{ideal triangulation} of $\surf$ is any maximal collection of pairwise compatible arcs whose relative interiors do not intersect each other.
\end{enumerate}
\end{defi}

The next proposition says that the pairwise compatibility of any collection of arcs can be simultaneously realized.

\begin{prop}\cite{FHS} Given any collection of pairwise compatible arcs, it is always possible to find representatives in their isotopy classes whose relative interiors do not intersect each other.
\end{prop}

\begin{defi}\label{def:types-of-ideal-triangles} Let $\tau$ be an ideal triangulation of a surface $\surf$.
\begin{enumerate}\item For each connected component of the complement in $\surfnoM$ of the union of the arcs in $\tau$, its topological closure $\triangle$ will be called an \emph{ideal triangle} of $\tau$.
\item An ideal triangle $\triangle$ is called \emph{interior} if its intersection with the boundary of $\surfnoM$ consists only of (possibly none) marked points. Otherwise it will be called \emph{non-interior}.
\item An interior ideal triangle $\triangle$ is \emph{self-folded} if it contains exactly two arcs of $\tau$ (see Figure \ref{fig:selffoldedtriang}).
\end{enumerate}
\end{defi}

        \begin{figure}[!h]
                \caption{Self-folded triangle}\label{fig:selffoldedtriang}
                \centering
                \includegraphics[scale=.4]{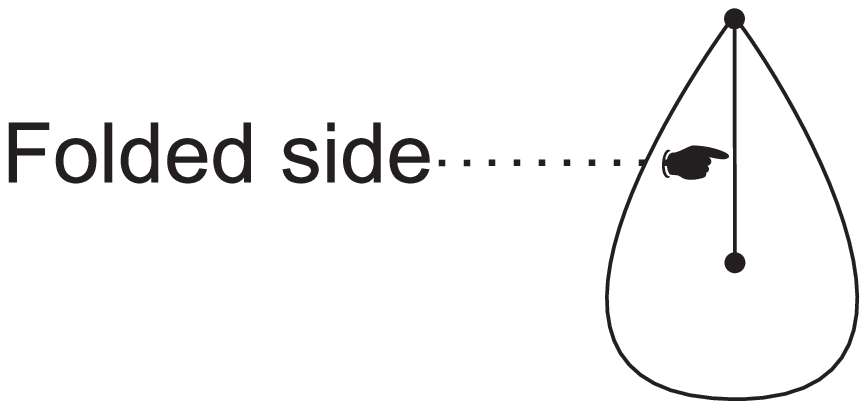}
        \end{figure}

The number $n$ of arcs in an ideal triangulation of $\surf$ is determined by the genus $g$ of $\surfnoM$, the number $b$ of boundary components
of $\surfnoM$, the number $p$ of punctures and the number $c$ of marked points on the boundary of $\surfnoM$, according to the formula
$n=6g+3b+3p+c-6$, which can be proved using the definition and basic properties of the Euler characteristic. Hence $n$ is an invariant of
$\surf$, called the \emph{rank} of $\surf$ (because it coincides with the rank of the cluster algebra associated to $\surf$, see \cite{FST}).

Let $\tau$ be an ideal triangulation of $\surf$ and let $\arc\in\tau$ be an arc. If $\arc$ is not the folded side of a self-folded triangle,
then there exists exactly one arc $\arc'$, different from $\arc$, such that $\sigma=(\tau\setminus\{\arc\})\cup\{\arc'\}$ is an ideal
triangulation of $\surf$. We say that $\sigma$ is obtained by applying a \emph{flip} to $\tau$, or by \emph{flipping} the arc $\arc$, and write
$\sigma=f_i(\tau)$.
In order to be able to flip the folded sides of self-folded triangles, we have to enlarge the set of arcs with which
triangulations are formed. This is done by introducing the notion of \emph{tagged arc}.

\begin{defi} A \emph{tagged arc} in $\surf$ is an ordinary arc together with a tag accompanying each of its two ends,
constrained to the following four conditions:
\begin{itemize}
\item a tag can only be \emph{plain} or \emph{notched};
\item the arc does not cut out a once-punctured monogon;
\item every end corresponding to a marked point that lies on the boundary must be tagged plain;
\item both ends of a loop must be tagged in the same way.
\end{itemize}
Note that there are arcs whose ends may be tagged in different ways. Following \cite{FST}, in the figures we will omit the plain tags and represent the notched ones by the symbol $\bowtie$. There is a straightforward way to extend the notion of isotopy to tagged arcs. We denote by $\taggedinsurf$ the set of (isotopy classes of) tagged arcs in $\surf$.
\end{defi}

If no confusion is possible, we will often refer to ordinary arcs simply as arcs. However, the word ``tagged" will never be dropped from the term ``tagged arc". Notice that not every ordinary arc is a tagged arc: a loop that encloses a once-punctured monogon is not a tagged arc.

\begin{defi}\begin{enumerate}\item Two tagged arcs $\arc_1$ and $\arc_2$ are \emph{compatible} if the following 
conditions are satisfied:
\begin{itemize}
\item the untagged versions of $\arc_1$ and $\arc_2$ are compatible as ordinary arcs;
\item if the untagged versions of $\arc_1$ and $\arc_2$ are different, then they are tagged in the same way at each end they share.
\item if the untagged versions of $\arc_1$ and $\arc_2$ coincide, then there must be at least one end of the untagged version at which they are tagged in the same way.
\end{itemize}
\item A \emph{tagged triangulation} of $\surf$ is any maximal collection of pairwise compatible tagged arcs.
\end{enumerate}
\end{defi}

All tagged triangulations of $\surf$
have the same cardinality (equal to the rank $n$ of $\surf$) and every collection of $n-1$ pairwise compatible tagged arcs is contained in precisely two tagged triangulations. This means that every tagged arc in a tagged triangulation can be replaced by a uniquely defined, different tagged arc that together with the remaining $n-1$ arcs forms a tagged triangulation. By analogy with the ordinary case, this combinatorial replacement will be called \emph{flip}. Furthermore, a sequence $(\tau_0,\ldots,\tau_\ell)$ of ideal or tagged triangulations will be called a \emph{flip-sequence} if $\tau_{k-1}$ and $\tau_k$ are related by a flip for $k=1,\ldots,\ell$.

\begin{prop}
\label{prop:ideal-triangs-seqs-of-flips}\label{seqofflips} Let $\surf$ be a surface with non-empty boundary.
\begin{itemize}\item Any two ideal triangulations of $\surf$ are members of a flip-sequence that involves only ideal triangulations.
\item Any two ideal triangulations without self-folded triangles are members of a flip-sequence that involves only ideal triangulations without self-folded triangles.
\item Any two tagged triangulations are members of a flip-sequence.
\end{itemize}
\end{prop}

The first assertion of Proposition \ref{prop:ideal-triangs-seqs-of-flips} is well known and has many different proofs, we refer the reader to \cite{Mosher} for an elementary one. The second assertion of the proposition is proved in \cite{Lreps}. A proof of the third assertion can be found in \cite{FST}.

Before defining signed adjacency matrices (and quivers) of tagged triangulations, let us see how to represent ideal triangulations with tagged ones and viceversa.

\begin{defi}\label{def:tagfunction(ideal-arc)} Let $\epsilon:\punct\rightarrow\{-1,1\}$ be any function. We define a function $\tagfunction_\epsilon:\arcsinsurf\rightarrow\taggedinsurf$ that represents ordinary arcs by tagged ones as follows.
\begin{enumerate}\item If $\arc$ is an ordinary arc that is not a loop enclosing a once-punctured monogon, set $\arc$ to be the underlying ordinary arc of the tagged arc $\tagfunction_\epsilon(\arc)$. An end of $\tagfunction_\epsilon(\arc)$ will be tagged notched if and only if the corresponding marked point is an element of $\punct$ where $\epsilon$ takes the value $-1$.
\item If $\arc$ is a loop, based at a marked point $m$, that encloses a once-punctured monogon, being $p$ the puncture inside this monogon, then the underlying ordinary arc of $\tagfunction(\arc)$ is the arc that connects $m$ with $p$ inside the monogon. The end at $m$ will be tagged notched if and only if $m\in\punct$ and $\epsilon(m)=-1$, and the end at $p$ will be tagged notched if and only if $\epsilon(p)=1$.
\end{enumerate}
In the case where $\epsilon$ is the function that takes the value $1$ at every puncture, we shall denote the corresponding function $\tagfunction_\epsilon$ simply by $\tagfunction$.
\end{defi}

To pass from tagged triangulations to ideal ones we need the notion of signature.

\begin{defi}\label{def:tau^circ} \begin{enumerate}\item Let $\tau$ be a tagged triangulation of $\surf$. The \emph{signature} of $\tau$ is the function
$\delta_\tau:\punct\rightarrow\{-1,0,1\}$ defined by
\begin{equation}
\delta_\tau(p)=
\begin{cases} 1\ \ \ \ \ \text{if all ends of tagged arcs in $\tau$ incident to $p$ are tagged plain;}\\ -1\ \ \ \text{if all ends of tagged arcs in $\tau$ incident to $p$ are tagged notched;}\\
0\ \ \ \ \ \text{otherwise.}
\end{cases}
\end{equation}
Note that if $\delta_\tau(p)=0$, then there are precisely two arcs in $\tau$ incident to $p$, the untagged versions of these arcs coincide and they carry the same tag at the end different from $p$.
\item We replace each tagged arc in $\tau$ with an ordinary arc by means of the following rules:
\begin{itemize}
\item delete all tags at the punctures $p$ with non-zero signature
\item for each puncture $p$ with $\delta_\tau(p)=0$, replace the tagged arc $\arc\in\tau$ which is notched at $p$ by a loop enclosing $p$ and $\arc$. 
    The resulting collection of ordinary arcs will be denoted by $\tau^\circ$.
\end{itemize}
\end{enumerate}
\end{defi}


The next proposition follows from the results in Subsection 9.1 of \cite{FST}.

\begin{prop}\label{prop:tagfunction-and-circ} Let $\surf$ be a surface and $\epsilon:\punct\rightarrow\{-1,1\}$ be a function.
\begin{itemize}
\item The function $\tagfunction_\epsilon:\arcsinsurf\to\taggedinsurf$ is injective and preserves compatibility. Thus, if $\arc_1$ and $\arc_2$ are compatible ordinary arcs, then $\tagfunction_\epsilon(\arc_1)$ and $\tagfunction_\epsilon(\arc_2)$ are compatible tagged arcs. Consequently, if $T$ is an ideal triangulation of $\surf$, then $\tagfunction_\epsilon(T)=\{\tagfunction_\epsilon(\arc)\suchthat\arc\in T\}$ is a tagged triangulation of $\surf$. Moreover, if $T_1$ and $T_2$ are ideal triangulations such that $T_2=f_\arc(T_1)$ for an arc $\arc\in T_1$, then $\tagfunction_\epsilon(T_2)=f_{\tagfunction_\epsilon(\arc)}(\tagfunction_{\epsilon}(T_1))$.
\item If $\tau$ is a tagged triangulation of $\surf$, then $\tau^\circ$ is an ideal triangulation of $\surf$.
\item For every ideal triangulation $T$, we have $\tagfunction(T)^\circ=T$.
\item For any tagged triangulation $\tau$ such that $\delta_\tau(p)\epsilon(p)\geq 0$ for every $p\in\punct$, $\tagfunction_\epsilon(\tau^\circ)=\tau$. Consequently, $\tagfunction(\tau^\circ)$ can be obtained from $\tau$ by deleting the tags at the punctures with negative signature.
\item Let $\tau$ and $\sigma$ be tagged triangulations such $\delta_\tau(p)\epsilon(p)\geq 0$ and $\delta_\sigma(p)\epsilon(p)\geq 0$ for every $p\in\punct$. Assume further that every puncture $p$ with $\delta_{\tau}(p)\delta_{\sigma}(p)=0$ satisfies $\epsilon(p)=1$. If $\sigma=f_{\arc}(\tau)$ for a tagged arc $\arc\in\tau$, then $\sigma^\circ=f_{\arc^\circ}(\tau^\circ)$, where $\arc^\circ\in\tau^\circ$ is the ordinary arc that replaces $\arc$ in Definition \ref{def:tau^circ}.
\end{itemize}
\end{prop}

%

\begin{defi}\label{def:Omega-prime}
For each function $\epsilon:\punct\rightarrow\{-1,1\}$, let $\bar{\Omega}'_\epsilon=\{\tau\suchthat\tau$ is a tagged triangulation of $\surf$ such that for all $p\in\punct$, $\delta_{\tau}(p)=-1$ if and only if $\epsilon(p)=-1\}$.
\end{defi}

\begin{remark}\label{rem:tagfunction-and-circ} \begin{enumerate}\item
There may exist non-compatible ordinary arcs $\arc_1$ and $\arc_2$ such that $\tagfunction_\epsilon(\arc_1)$ and $\tagfunction_\epsilon(\arc_2)$ are compatible.
\item Note that every tagged triangulation belongs to exactly one set $\bar{\Omega}'_\epsilon$.
\item If $\punct\neq\varnothing$, and $\epsilon$ takes the value $-1$ on at least one puncture, then the set $\bar{\Omega}'_\epsilon$ is properly contained in the \emph{closed stratum} $\bar{\Omega}_\epsilon$ defined by Fomin-Shapiro-Thurston. Our choice of notation $\bar{\Omega}'_\epsilon$ is made with the purpose of avoiding confusion with the closed stratum.
\end{enumerate}
\end{remark}

To each ideal triangulation $T$ we associate a skew-symmetric $n\times n$ integer matrix $B(T)$ whose rows and columns correspond to the
arcs of $T$. Let $\pi_T:T\rightarrow T$ be the function that is the identity on the set of arcs
that are not folded sides of self-folded triangles of $T$, and sends the folded side of a self-folded triangle to the unique loop of $T$
enclosing it. For each non-self-folded ideal triangle $\triangle$ of $T$, let $B^\triangle=b^\triangle_{ij}$ be the $n\times n$ integer
matrix defined by
\begin{equation}
b^\triangle_{ij}=
\begin{cases} 1\ \ \ \ \ \text{if $\triangle$ has sides $\pi_T(i)$ and $\pi_T(j)$, with $\pi_T(j)$ preceding
$\pi_T(i)$}\\ \ \ \ \ \ \ \ \ \ \text{in the clockwise order defined by the orientation of $\surfnoM$;}\\ -1\ \ \ \text{if the same holds,
but in the counter-clockwise
order;}\\
0\ \ \ \ \ \text{otherwise.}
\end{cases}
\end{equation}
The \emph{signed adjacency matrix} $B(T)$ is then defined as
\begin{equation}
B(T)=\underset{\triangle}{\sum}B^\triangle,
\end{equation}
where the sum runs over all non-self-folded triangles of $T$. For a tagged triangulation $\tau$, the signed adjacency matrix is defined as $B(\tau)=B(\tau^\circ)$, with its rows and columns labeled by the tagged arcs in $\tau$, rather than the arcs in $\tau^\circ$.

Note that all entries of $B(\tau)$ have absolute value less than 3. Moreover, $B(\tau)$ is skew-symmetric, hence gives rise to the \emph{signed adjacency quiver} $\qtau$, whose vertices are the tagged arcs in $\tau$, with $b_{ij}$ arrows from $j$ to $i$ whenever $b_{ij}>0$. Since $B(\tau)$ is skew-symmetric, $\qtau$ is a 2-acyclic quiver.

\begin{thm}\cite[Proposition 4.8 and Lemma 9.7]{FST}
Let $\tau$ and $\sigma$ be tagged triangulations. If $\sigma$ is obtained from $\tau$ by flipping the tagged arc $\arc$ of $\tau$, then $Q(\sigma)=\mu_\arc(\qtau)$.
\end{thm}


For technical reasons, we introduce some quivers that are obtained from signed adjacency quivers by adding some 2-cycles in specific situations.

\begin{defi} Let $\tau$ be a tagged triangulation of $\surf$. For each puncture $p$ incident to exactly two tagged arcs of $\tau$ that have the same tag at $p$, we add to $\qtau$ a 2-cycle that connects those tagged arcs and call the resulting quiver the \emph{unreduced signed adjacency quiver} $\unredqtau$. 
\end{defi}

It is clear that $\qtau$ can always be obtained from $\unredqtau$ by deleting all 2-cycles.


\section{The QP of a tagged triangulation}\label{sec:QPs-of-tagged-triangulations}

Let $\tau$ be a tagged triangulation of $\surf$ and $\epsilon:\punct\rightarrow\{-1,1\}$ be the unique function that takes the value $-1$ precisely at the punctures where the signature of $\tau$ is negative. A quick look at Definitions \ref{def:tagfunction(ideal-arc)} and \ref{def:tau^circ} tells us that $\tagfunction_\epsilon$ restricts to a bijection $\tau^\circ\rightarrow\tau$. This bijection is actually a quiver isomorphism between $Q(\tau^\circ)$ and $\qtau$. It therefore induces an $R$-algebra isomorphism $R\langle\langle Q(\tau^\circ)\rangle\rangle\rightarrow R\langle\langle\qtau\rangle\rangle$, which we shall denote by $\tagfunction_\epsilon$ as well.

The following definition generalizes Definition 23 of \cite{Lqps} to arbitrary tagged triangulations (of surfaces with non-empty boundary). We start by fixing a choice $(x_p)_{p\in\punct}$ of non-zero elements of the ground field $K$.

\begin{defi}\label{def:QP-of-tagged-triangulation} Let $\surf$ be a surface with non-empty boundary.
\begin{enumerate}\item If $\tau$ is an ideal triangulation of $\surf$, we define the potentials $\unredstau$ and $\stau$ according to the following rules:
\begin{itemize} \item Each interior non-self-folded ideal triangle $\triangle$ of $\tau$ gives rise to an oriented triangle of $\unredqtau$, let $\widehat{S}^\triangle$ be such oriented triangle up to cyclical equivalence.
\item If the interior non-self-folded ideal triangle $\triangle$ with sides $j$, $k$ and $l$, is adjacent to two self-folded triangles like in the configuration of Figure \ref{adsftriangs},
        \begin{figure}[!h]
                \caption{}\label{adsftriangs}
                \centering
                \includegraphics[scale=.5]{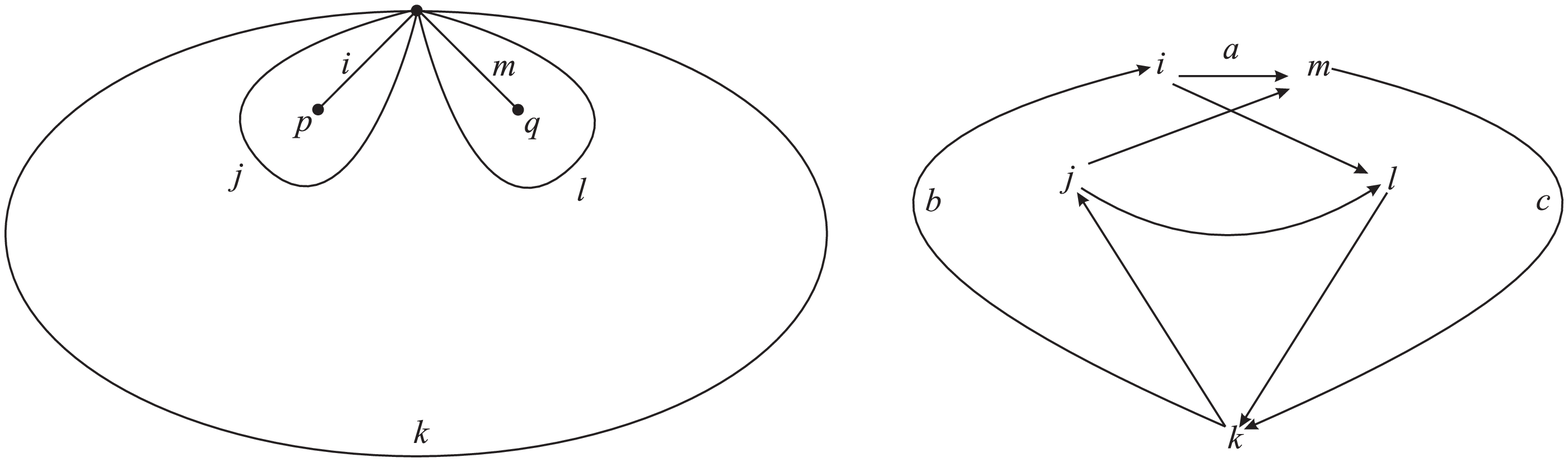}
        \end{figure}\\
define $\widehat{U}^\triangle=\frac{abc}{x_px_q}$ (up to cyclical equivalence), where $p$ and $q$ are the punctures enclosed in the self-folded triangles adjacent to $\triangle$. Otherwise, if it is adjacent to less than two self-folded triangles, define $\widehat{U}^\triangle=0$.
\item If a puncture $p$ is adjacent to exactly one arc $\arc$ of $\tau$, then $\arc$ is the folded side of a self-folded triangle of $\tau$ and around $\arc$ we have the configuration shown in Figure \ref{sftriangle}.
        \begin{figure}[!h]
                \caption{}\label{sftriangle}
                \centering
                \includegraphics[scale=.5]{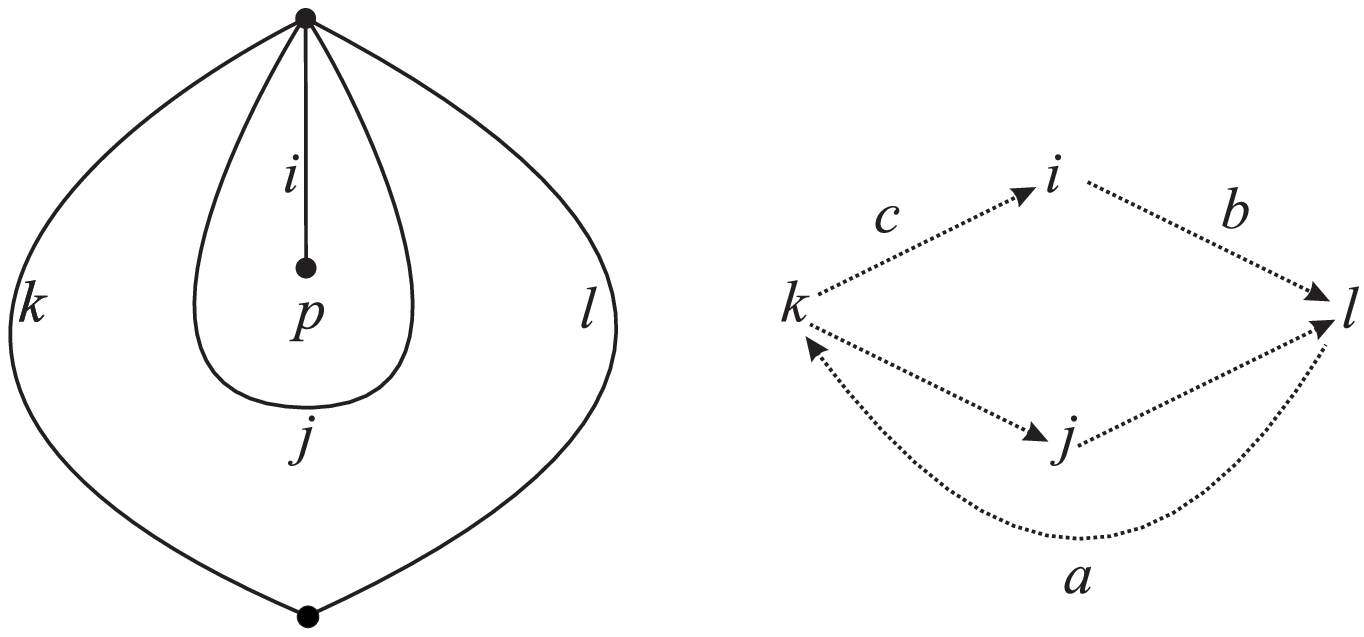}
        \end{figure}\\
In case both $k$ and $l$ are indeed arcs of $\tau$ (and not part of the boundary of $\surfnoM$), then we define $\widehat{S}^p=-\frac{abc}{x_p}$ (up to cyclical equivalence).
\item If a puncture $p$ is adjacent to more than one arc, delete all the loops incident to $p$ that enclose self-folded triangles. The arrows between the remaining arcs adjacent to $p$ form a unique cycle $a^p_1\ldots a^p_{d_p}$, without repeated arrows, that exhausts all such remaining arcs and gives a complete round around $p$ in the counter-clockwise orientation defined by the orientation of $\surfnoM$. We define $\widehat{S}^p=x_pa^p_1\ldots a^p_{d_p}$ (up to cyclical equivalence).
\end{itemize}
The \emph{unreduced potential} $\unredstau\in R\langle\langle\unredqtau\rangle\rangle$ of $\tau$ is then defined by
\begin{equation}
\unredstau=\underset{\triangle}{\sum}(\widehat{S}^\triangle+\widehat{U}^\triangle)+\underset{p\in\punct}{\sum}\widehat{S}^p,
\end{equation}
where the first sum runs over all interior non-self-folded triangles. We define $\qstau$ to be the (right-equivalence class of the) reduced part of $\unredqstau$.
\item If $\tau$ is a tagged triangulation of $\surf$, we define $\unredstau=\tagfunction_\epsilon(\widehat{S}(\tau^\circ))$ and $\stau=\tagfunction_\epsilon(S(\tau^\circ))$, where $\epsilon:\punct\rightarrow\{-1,1\}$ is the unique function that takes the value $1$ at all the punctures where the signature of $\tau$ is non-negative and the value $-1$ at all the punctures where the signature of $\tau$ is negative.
\end{enumerate}
\end{defi}

\begin{remark}\label{rem:hooks-only-once}\begin{enumerate}\item Let $\tau$ be an ideal triangulation. The only situation where one needs to apply reduction to $(\unredqtau,\unredstau)$ in order to obtain $\stau$ is when there is some puncture incident to exactly two arcs of $\tau$. The reduction is done explicitly in \cite[Section 3]{Lqps}.
 \item Let $\tau$ be an ideal triangulation and $\arc\in\tau$ an arc that is not the folded side of any self-folded triangle of $\tau$. If $\beta\gamma$ is an $\arc$-hook of $\qtau$ and there is some oriented 3-cycle $\alpha\beta\gamma$ of $\qtau$ that appears in $\stau$ with non-zero coefficient, then $\beta\gamma$ is not factor of any other term of $\stau$.
 \end{enumerate}
\end{remark}

Since $\tagfunction_\epsilon$ is a `relabeling' of vertices (it `relabels' the elements of $\tau^\circ$ with the corresponding tagged arcs in $\tau=\tagfunction_\epsilon(\tau^\circ)$), we have the following.

\begin{lemma} For every tagged triangulation $\tau$, the QP $\qstau$ is the reduced part of $\unredqstau$.
\end{lemma}

We arrive at our first main result, which says that any two tagged triangulations are connected by a flip sequence that is compatible with QP-mutation.

\begin{thm}\label{thm:flip<->mutation-tagged-triangs} Let $\surf$ be a surface with non-empty boundary and $\tau$ and $\sigma$ be tagged triangulations of $\surf$. Then there exists a flip sequence $(\tau_0,\tau_1,\ldots,\tau_\ell)$ such that:
\begin{itemize}\item $\tau_0=\tau$ and $\tau_\ell=\sigma$;
\item $\mu_{\arc_k}(Q(\tau_{k-1}),S(\tau_{k-1}))$ is right-equivalent to $(Q(\tau_k),S(\tau_k))$ for $k=1,\ldots,\ell$, where
$\arc_{k}$ denotes the tagged arc such that $\tau_k=f_{\arc_k}(\tau_{k-1})$.
\end{itemize}
\end{thm}

The proof of Theorem \ref{thm:flip<->mutation-tagged-triangs} will make use of the following.

\begin{thm}\cite[Theorems 30, 31 and 36]{Lqps}\label{thm:flip<->mutation-ideal-triangs} Let $\surf$ be a surface with non-empty boundary. If $\tau$ and $\sigma$ are ideal triangulations of $\surf$ with $\sigma=f_\arc(\tau)$, then $\mu_\arc(\qstau)$ and $\qssigma$ are right-equivalent QPs. Furthermore, all QPs of the form $\qtau$, for $\tau$ an ideal triangulation of $\surf$, are Jacobi-finite and non-degenerate.
\end{thm}

\begin{proof}[Proof of Theorem \ref{thm:flip<->mutation-tagged-triangs}] Our strategy is the following:
\begin{enumerate}\item\label{proof:item-1} First we will show that for each pair of functions $\epsilon_1,\epsilon_2:\punct\rightarrow\{-1,1\}$ such that $\sum_{p\in\punct}|\epsilon_1(p)-\epsilon_2(p)|=2$, there exist tagged triangulations $\tau'\in\bar{\Omega}'_{\epsilon_1}$, $\sigma'\in\bar{\Omega}'_{\epsilon_2}$ (see Definition \ref{def:Omega-prime}), with $\sigma'=f_{\arc'}(\tau')$ for some tagged arc $\arc'$, such that $\mu_{\arc'}(Q(\tau'),S(\tau'))$ and $(Q(\sigma'),S(\sigma'))$ are right-equivalent.
\item\label{proof:item-2} Then we will prove that for fixed $\epsilon:\punct\rightarrow\{-1,1\}$, any two elements of $\bar{\Omega}'_\epsilon$ are related by a sequence of flips of tagged triangulations belonging to $\bar{\Omega}'_\epsilon$.
\item\label{proof:item-3} Based on Theorem \ref{thm:flip<->mutation-ideal-triangs}, we will then show that if $\tau$ and $\sigma$ belong to the same set $\bar{\Omega}'_\epsilon$ and are related by a single flip, then their QPs $\qstau$ and $\qssigma$ are related by the corresponding QP-mutation.
\end{enumerate}

\begin{lemma}\label{lemma:change-of-epsilon} If $\epsilon_1,\epsilon_2:\punct\rightarrow\{-1,1\}$ are functions satisfying $\sum_{p\in\punct}|\epsilon_1(p)-\epsilon_2(p)|=2$, then there exist tagged triangulations $\tau'\in\bar{\Omega}'_{\epsilon_1}$, $\sigma'\in\bar{\Omega}'_{\epsilon_2}$, with $\sigma'=f_{\arc'}(\tau')$ for some tagged arc $\arc'$, such that $\mu_{\arc'}(Q(\tau'),S(\tau'))$ and $(Q(\sigma'),S(\sigma'))$ are right-equivalent.
\end{lemma}

\begin{proof} Throughout the proof of this lemma we will make a slight change to our notation. Specifically, we will assume that the set $\marked$ is contained in the (non-empty) boundary of $\surfnoM$, so that $\surf$ is an unpunctured surface. Then we are going to add punctures to $\surf$ one by one, and denote the set of marked points of the resulting $n$-punctured surface by $\marked\cup\punct_n=\marked\cup\{p_1,\ldots,p_n\}$ (for $n\geq 0$, where $\punct_0=\varnothing$). The alluded punctures will be added at the same time that we recursively construct some specific ideal triangulations of the corresponding punctured surfaces.

Let $\tau_0=\sigma_0$ be any ideal triangulation of the unpunctured surface $\surf$. Since the boundary of $\surfnoM$ is not empty, this ideal triangulation must have a non-interior triangle. Put a puncture $p_1$ inside any such triangle $\triangle_0$. Then draw the three arcs emanating from $p_1$ and going to the three
vertices of $\triangle_0$. The result is an ideal triangulation $\sigma_1$ of $\surfpone$.

For $n>1$, once $\sigma_{n-1}$ has been constructed, we put a puncture $p_n$ inside a non-interior
triangle $\triangle_{n-1}$ of $\sigma_{n-1}$ having $p_{n-1}$ as a vertex. Then we draw the three arcs emanating from $p_n$ and going to the
three vertices of $\triangle_{n-1}$. The result is an ideal triangulation $\sigma_n$ of $\surfpn$.

We have thus recursively constructed a sequence $\sigma_0,\sigma_1,\ldots$ of ideal triangulations with the property that $\sigma_n$ is a triangulation of $(\surfnoM,\marked\cup\punct_n)$ for each $n\geq 0$. Fix $n>0$ and a non-interior ideal triangle $\triangle_n$ of $\sigma_n$ having $p_n$ as a vertex. Then $\triangle_n$ has exactly one side which is a boundary segment of $\surfnoM$, and its two remaining sides are arcs in $\sigma_n$ incident to $p_n$. We denote these two arcs by $j^n_1$ and $j^n_3$ in the counterclockwise direction around $p_n$. Notice that, by definition of $\sigma_n$, there is exactly one arc in $\sigma_n$, different from $j^n_1$ and $j^n_3$, that is incident to $p_n$. We denote this arc by $j^n_2$, and define $\tau_n=f_{j_2^n}(\sigma_n)$ (see Figure \ref{Fig:f1f2f3}).
        \begin{figure}[!h]
                \caption{$\sigma_n$ and $\tau_n$ for $n\geq 0$}\label{Fig:f1f2f3}
                \centering
                \includegraphics[scale=.6]{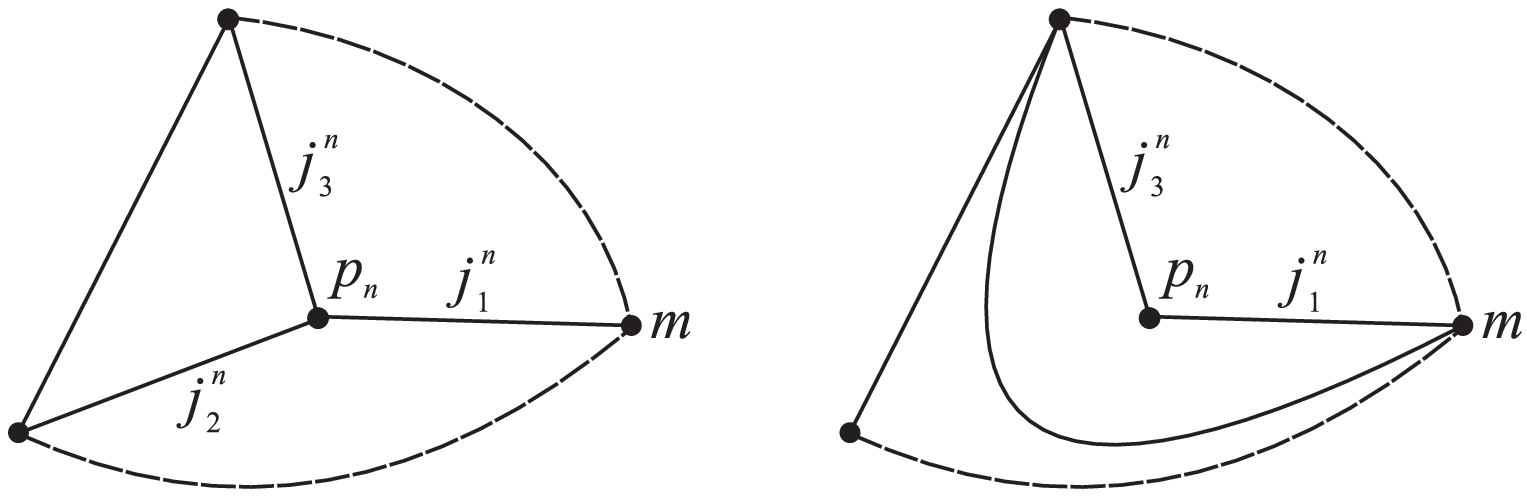}
        \end{figure}
Notice that $\tau_n$ does not have self-folded triangles.

Now, let $\epsilon_1,\epsilon_2:\punct_n\rightarrow\{-1,1\}$ be functions satisfying $\sum_{l=1}^n|\epsilon_1(p_l)-\epsilon_2(p_l)|=2$. This means that there exists exactly one puncture $p_k\in\punct_n$ such that $\epsilon_1(p_k)\neq\epsilon_2(p_k)$. Without loss of generality, we can suppose that $p_k=p_n$ and $\epsilon_1(p_n)=-1=-\epsilon_2(p_n)$. The tagged triangulations $\tau'=\tagfunction_{\epsilon_1}(\tau_n)$ and $\sigma'=f_{\arc'}(\tau')$, where $\arc'=\tagfunction_{\epsilon_1}(j_3^n)$ (see Figure \ref{Fig:tausigma}), certainly satisfy $\tau'\in\bar{\Omega}'_{\epsilon_1}$ and $\sigma'\in\bar{\Omega}'_{\epsilon_2}$.
        \begin{figure}[!h]
                \caption{$\tau'$ and $\sigma'$}\label{Fig:tausigma}
                \centering
                \includegraphics[scale=.6]{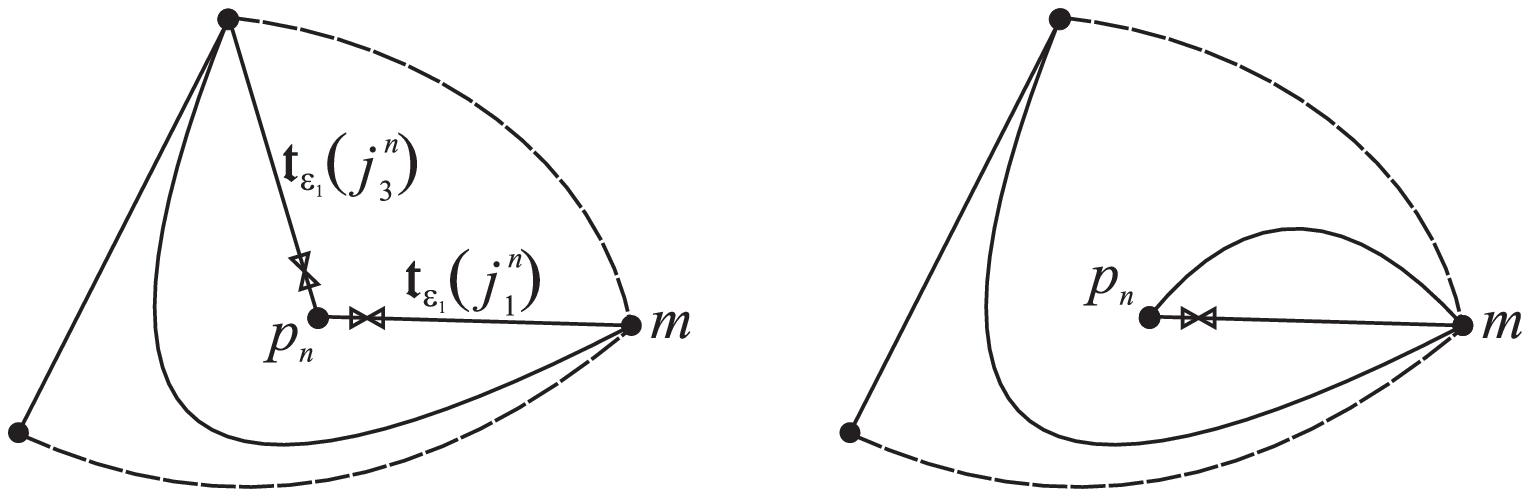}
        \end{figure}
It is obvious that $\mu_{\arc'}(Q(\tau'),S(\tau'))$ is right-equivalent to $(Q(\sigma'),S(\sigma'))$, for $\arc'$ is a sink of the quiver $Q(\tau')$. Lemma \ref{lemma:change-of-epsilon} is proved.
\end{proof}

\begin{lemma}\label{lemma:Omega'epsilon-is-connected} Fix a function $\epsilon:\punct\rightarrow\{-1,1\}$. Any two distinct elements of $\bar{\Omega}'_\epsilon$ are related by a sequence of flips of tagged triangulations belonging to $\bar{\Omega}'_\epsilon$.
\end{lemma}

\begin{proof} The lemma is a consequence of the second assertion of Proposition \ref{prop:ideal-triangs-seqs-of-flips} and the following obvious fact:
\begin{equation}\label{eq:eliminate-zero-signature-puncts}
\text{Any element of $\bar{\Omega}'_\epsilon$ either is a tagged triangulation without zero-signature punctures, or can be}
\end{equation}
\begin{center}
transformed to one such by a sequence of flips that involve only tagged triangulations belonging to $\bar{\Omega}'_\epsilon$.
\end{center}
Indeed, let $\tau$ and $\sigma$ be two distinct elements of $\bar{\Omega}'_\epsilon$. By \eqref{eq:eliminate-zero-signature-puncts} there are flip-sequences
$(\tau_0,\tau_{1},\ldots,\tau_{n})$, $(\sigma_{0},\ldots,\sigma_{m})$, of tagged triangulations belonging to $\bar{\Omega}'_\epsilon$, such that $\tau_0=\tau$, $\sigma_0=\sigma$, and none of $\tau_n$ and $\sigma_m$ has zero-signature punctures. Consequently, none of the ideal triangulations $\tau_n^\circ$ and $\sigma_m^\circ$ has self-folded triangles. By Proposition \ref{prop:ideal-triangs-seqs-of-flips}, there is a flip-sequence
$(T_0,\ldots,T_{l})$ involving only ideal triangulations without self-folded triangles, such that $T_0=\tau_n^\circ$ and $T_l=\sigma_m^\circ$. By the first statement of Proposition \ref{prop:tagfunction-and-circ}, if we apply $\tagfunction_\epsilon$ to each of the ideal triangulations
$T_0,\ldots,T_l$, we obtain a flip-sequence $(\tagfunction_\epsilon(T_0),\ldots,\tagfunction_\epsilon(T_{l}))$
of tagged triangulations belonging to $\bar{\Omega}'_\epsilon$. We conclude that $(\tau_0,\tau_{1},\ldots,\tau_{n},\tagfunction_\epsilon(T_1),\ldots,\tagfunction_\epsilon(T_{l-1}),\sigma_{m},\ldots,
\sigma_{0})$ is a flip-sequence of elements of $\bar{\Omega}'_\epsilon$ connecting $\tau$ with $\sigma$.
\end{proof}

\begin{lemma}\label{lemma:flip<->mutation-for-same-epsilon} If $\tau$ and $\sigma$ are tagged triangulations that belong to the same set $\bar{\Omega}'_\epsilon$ and are related by a single flip, then their QPs $\qstau$ and $\qssigma$ are related by the corresponding QP-mutation.
\end{lemma}

\begin{proof} This follows from the first and last statements of Proposition \ref{prop:tagfunction-and-circ} and the fact that $\qstau$ (resp. $\qssigma$) is obtained from $(Q(\tau^\circ),S(\tau^\circ))$ (resp. $(Q(\sigma^\circ),S(\sigma^\circ))$) by ``relabeling" the vertices of $\tau^\circ$ (resp. $\sigma^\circ$) using $\tagfunction_{\epsilon}$. Explicitly, let $\psi_\tau$ (resp. $\psi_\sigma$) denote the inverse of the $R$-algebra isomorphism
$\tagfunction_\epsilon:R\langle\langle Q(\tau^\circ)\rangle\rangle\rightarrow R\langle\langle Q(\tau)\rangle\rangle$ (resp. $\tagfunction_\epsilon:R\langle\langle Q(\sigma^\circ)\rangle\rangle\rightarrow R\langle\langle Q(\sigma)\rangle\rangle$). Suppose that $\arc$ is the arc in $\tau$ such that $\sigma=f_\arc(\tau)$. Then $\sigma^\circ=f_{\psi_\tau(\arc)}(\tau^\circ)$. This implies, by part (a) of Theorem \ref{thm:flip<->mutation-ideal-triangs}, that there exists a right-equivalence $\varphi:\mu_{\psi_{\tau}(\arc)}(Q(\tau^\circ),S(\tau^\circ))\longrightarrow(Q(\sigma^\circ),S(\sigma^\circ))$. The composition $\tagfunction_\epsilon\circ\varphi\circ\psi_\sigma:\mu_\arc\qstau\longrightarrow\qssigma$ is then a right-equivalence that proves Lemma \ref{lemma:flip<->mutation-for-same-epsilon}.
\end{proof}

To finish the proof of Theorem \ref{thm:flip<->mutation-tagged-triangs}, note that if $\tau$ and $\sigma$ are tagged triangulations related by a single flip, then the functions $\epsilon_\tau$ and $\epsilon_\sigma$ (such that $\tau\in\bar{\Omega}'_{\epsilon_\tau}$ and $\sigma\in\bar{\Omega}'_{\epsilon_\sigma}$) either are equal or satisfy
$$
\sum_{p\in\punct}|\epsilon_\tau(p)-\epsilon_\sigma(p)|=2.
$$
Thus, Theorem \ref{thm:flip<->mutation-tagged-triangs} follows from Lemmas \ref{lemma:change-of-epsilon}, \ref{lemma:Omega'epsilon-is-connected}, \ref{lemma:flip<->mutation-for-same-epsilon} and an easy induction. Indeed, let $\tau$ and $\sigma$ be arbitrary tagged triangulations of $\surf$, not necessarily related by a single flip. Then $\tau\in\bar{\Omega}'_{\epsilon_\tau}$ and $\sigma\in\bar{\Omega}'_{\epsilon_\sigma}$ for some functions $\epsilon_\tau,\epsilon_\sigma:\punct\to\{-1,1\}$. If $\epsilon_\tau=\epsilon_\sigma$, then Lemmas \ref{lemma:Omega'epsilon-is-connected} and \ref{lemma:flip<->mutation-for-same-epsilon} imply the existence of a flip-sequence from $\tau$ to $\sigma$ along which flips are compatible with QP-mutations. Otherwise, if $\epsilon_\tau\neq\epsilon_\sigma$, let $q\in\punct$ be such that $\epsilon_\tau(q)\neq\epsilon_\sigma(q)$. There certainly exists a tagged triangulation $\rho$ of $\surf$ such that $\{p\in\punct\suchthat\epsilon_\tau(p)\neq\epsilon_\rho(p)\}=\{q\}$, that is, such that $\sum_{p\in\punct}|\epsilon_\tau(p)-\epsilon_\rho(p)|=2$. Applying Lemmas \ref{lemma:change-of-epsilon} and \ref{lemma:flip<->mutation-for-same-epsilon} we obtain the existence of a flip-sequence from $\tau$ to $\rho$ along which flips are compatible with QP-mutations. On the other hand, the fact that $\{p\in\punct\suchthat\epsilon_\tau(p)\neq\epsilon_\rho(p)\}=\{q\}$ implies $|\{p\in\punct\suchthat\epsilon_\rho(p)\neq\epsilon_\sigma(p)\}|
<|\{p\in\punct\suchthat\epsilon_\tau(p)\neq\epsilon_\sigma(p)\}|$, and hence we can inductively assume that there exists a flip-sequence from $\rho$ to $\sigma$ along which flips are compatible with QP-mutations. This way we obtain a flip-sequence from $\tau$ to $\sigma$ along which flips are compatible with QP-mutations. This finishes the proof of Theorem \ref{thm:flip<->mutation-tagged-triangs}.
\end{proof}

Our second main result is the following corollary, which says that flip is compatible with mutation at least at the level of Jacobian algebras.

\begin{coro}\label{coro:flip<->mutation-jacobian} Let $\surf$ be a surface with non-empty boundary and $\tau$ and $\sigma$ be tagged triangulations of $\surf$. If $\sigma=f_\arc(\tau)$, then the Jacobian algebras $\mathcal{P}(\muti\qstau)$ and $\mathcal{P}\qssigma$ are isomorphic.
\end{coro}

Before proving this corollary, let us recall the cluster categorification of surfaces with non-empty boundary.
%

\begin{thm}\cite[Subsection 3.4]{Amiot-survey}\label{thm:cluster-category-for-surface} Let $\surf$ be a surface with non-empty boundary. There exists a $\Hom$-finite triangulated 2-Calabi-Yau category $\mathcal{C}_{\surf}$ with a cluster-tilting object $T_\tau$ associated to each tagged triangulation $\tau$ of $\surf$ in such a way that tagged triangulations related by a flip give rise to cluster-tilting objects related by the corresponding IY-mutation.\footnote{``IY" after Iyama-Yoshino.}
\end{thm}

In particular, the exchange graph of cluster-tilting objects reachable from any fixed $T_\tau$ coincides with $\mathbf{E}^{\bowtie}\surf$. (The vertices of $\mathbf{E}^{\bowtie}\surf$ are the tagged triangulations of $\surf$, and there is an edge connecting two tagged triangulations $\tau$ and $\sigma$ if and only if $\tau$ and $\sigma$ are related by the flip of a tagged arc. We omit the definition of the categorical concepts involved, and refer the reader to the papers \cite{Amiot-gldim2} and \cite{Amiot-survey} by Amiot.)

\begin{remark}\begin{enumerate}\item The existence of $\mathcal{C}_{\surf}$ follows by a combination of results from \cite{Amiot-gldim2} and \cite{Lqps}. More specifically, for each ideal triangulation $\tau$ of $\surf$, the QP $\qstau$ is non-degenerate and Jacobi-finite \cite{Lqps}, hence gives rise to a generalized cluster category \cite{Amiot-gldim2}. Since ideal triangulations related by a flip have QPs related by QP-mutation \cite{Lqps}, the cluster categories they give rise to are triangle-equivalent \cite{Amiot-gldim2}. Thus, $\mathcal{C}_{\surf}$ is defined to be the generalized cluster category of any of the QPs associated to ideal triangulations.
\item The fact that all tagged triangulations (and not only ideal ones) have cluster-tilting objects associated in $\mathcal{C}_{\surf}$ is independent of the results of the present paper. Indeed, Fomin-Shapiro-Thurston have proved that all cluster algebras associated to a signed adjacency quiver $Q(\tau)$ arising from a surface with non-empty boundary have the same exchange graph $\mathbf{E}^{\bowtie}\surf$. Thus, the fact that tagged triangulations have cluster-tilting objects associated to them is a consequence of the fact that the exchange graph of the principal coefficient cluster algebra of a quiver $Q$ coincides with the exchange graph of cluster-tilting objects IY-equivalent to the canonical cluster-tilting object of $\mathcal{C}_{(Q,S)}$, provided $S$ is a Jacobi-finite non-degenerate potential on $Q$.
\end{enumerate}
\end{remark}

\begin{proof}[Proof of Corollary \ref{coro:flip<->mutation-jacobian}] Take an arbitrary pair of tagged triangulations $\tau$ and $\sigma$ related by a single flip, say $\sigma=f_\arc(\tau)$. By Theorem \ref{thm:flip<->mutation-tagged-triangs}, there are flip-sequences $(\tau,\tau_1,\ldots,\tau_t)$ and $(\sigma,\sigma_1,\ldots,\sigma_s)$, such that $\tau_t$ and $\sigma_s$ are ideal triangulations, $\mu_{\arc_k}(Q(\tau_{k-1}),S(\tau_{k-1}))$ is right-equivalent to $(Q(\tau_{k}),S(\tau_k))$ for $k=1,\ldots,t$, and $\mu_{\arc_k}(Q(\sigma_{k-1}),S(\sigma_{k-1}))$ is right-equivalent to $(Q(\sigma_{k}),S(\sigma_k))$ for $k=1,\ldots,s$. Therefore, the endomorphism algebras of the cluster-tilting objects of $\mathcal{C}_{\surf}$ corresponding to $\tau$ and $\sigma$ are precisely the Jacobian algebras $\mathcal{P}\qstau$ and $\mathcal{P}\qssigma$. The result follows from the corresponding property of IY-mutation in $\Hom$-finite 2-Calabi-Yau triangulated categories.
\end{proof}

\begin{ex}\label{ex:3-punctured-hexagon} Consider the tagged triangulations $\tau$ and $\sigma=f_\arc(\tau)$ of the three-times-punctured hexagon shown in Figure \ref{Fig:hexagon_3_punct}, where the quivers $\qtau$ and $\qsigma$ are drawn as well.
        \begin{figure}[!h]
                \caption{}\label{Fig:hexagon_3_punct}
                \centering
                \includegraphics[scale=.45]{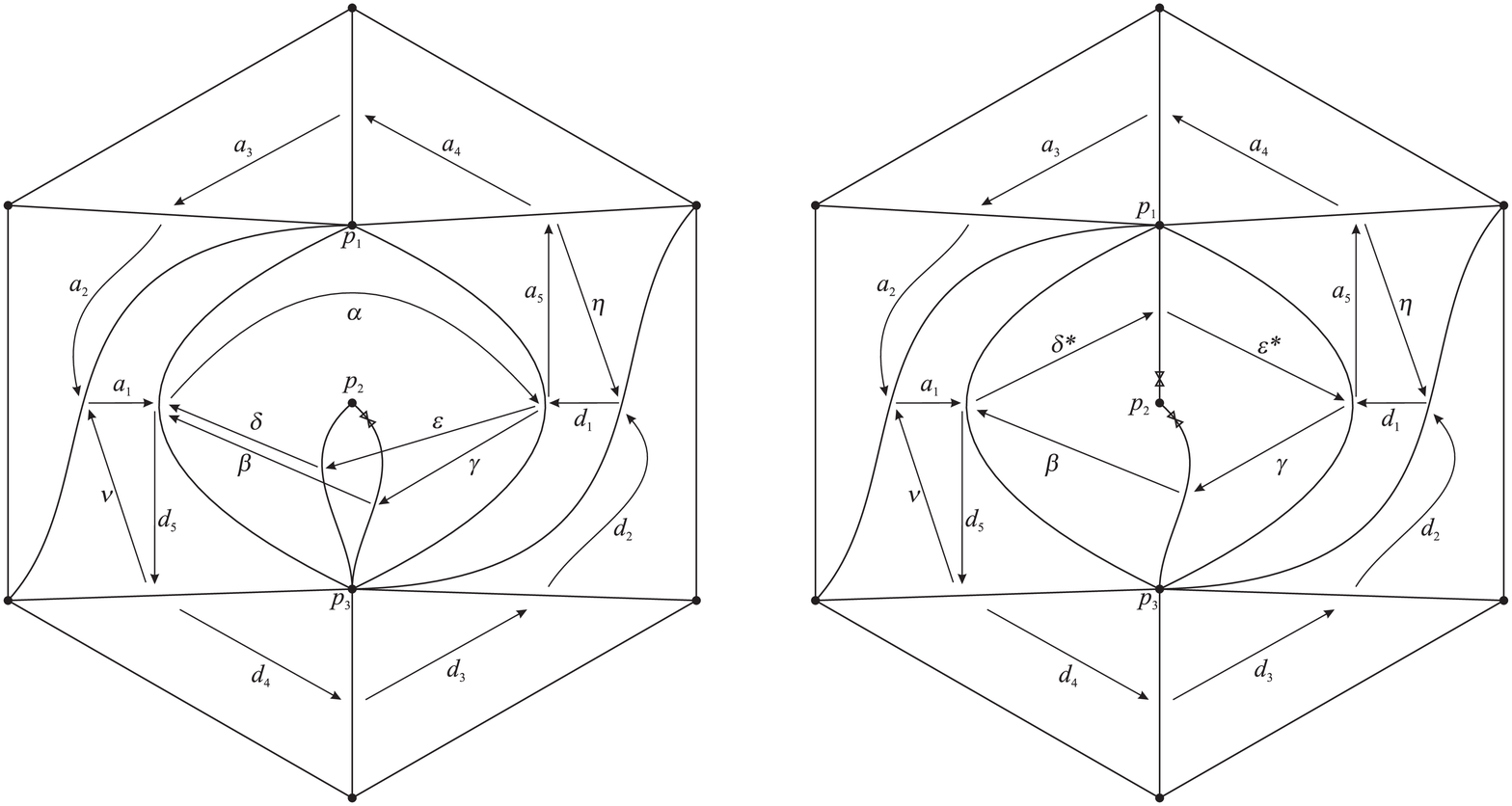}
        \end{figure}
As for the potentials, we have
$$
\stau=\alpha\beta\gamma+a_1\nu d_5+a_5d_1\eta+x_{p_1}\alpha a_1a_2a_3a_4a_5-x_{p_2}^{-1}\alpha\delta\varepsilon+x_{p_3}\delta\varepsilon d_1d_2d_3d_4d_5 \ \ \ \text{and}
$$
$$
\ssigma=a_1\nu d_5+ a_5d_1\eta+x_{p_1}\varepsilon^*\delta^*a_1a_2a_3a_4a_5-x_{p_2}^{-1}\varepsilon^*\delta^*\beta\gamma+x_{p_3}\beta\gamma d_1d_2d_3d_4d_5.
$$
If we apply the $\arc^{\operatorname{th}}$ QP-mutation to $\qstau$ we obtain the QP $(\mu_\arc(\qtau),\mu_\arc(\stau))$, where $\mu_\arc(\qtau)=\qsigma$ and
$$\mu_\arc(\stau)=a_1\nu d_5+a_5d_1\eta+x_{p_1}x_{p_2}\varepsilon^*\delta^* a_1a_2a_3a_4a_5+x_{p_2}\varepsilon^*\delta^*\beta\gamma
+x_{p_2}x_{p_3}\beta\gamma d_1d_2d_3d_4d_5+
$$
$$
+x_{p_1}x_{p_2}x_{p_3} d_1d_2d_3d_4d_5a_1a_2a_3a_4a_5.
$$
%
According to Corollary \ref{coro:flip<->mutation-jacobian}, the Jacobian algebras $\mathcal{P}\qssigma$ and
$\mathcal{P}(\mu_\arc(\qtau),$ $\mu_\arc(\stau))$ are isomorphic. Actually something stronger happens, namely, the $R$-algebra isomorphism $\varphi:R\langle\langle\qsigma\rangle\rangle\rightarrow R\langle\langle\mu_\arc(\qtau)\rangle\rangle$ given by
$$
a_1\mapsto-a_1, \ \ \ \nu \mapsto -\nu-x_{p_1}x_{p_2}x_{p_3}a_2a_3a_4a_5d_1d_2d_3d_4, \ \ \ \varepsilon^*\mapsto -x_{p_2}\varepsilon^*, \ \ \ \beta\mapsto x_{p_2}\beta,
$$
and the identity on the remaining arrows of $\qsigma$, is a right-equivalence $\qssigma\rightarrow(\mu_\arc(\qtau),\mu_\arc(\stau))$. So, the Jacobian algebras are isomorphic.
\end{ex}
%

\section{Admissibility of the Jacobian ideal}\label{sec:admissibility}

\begin{defi}[Admissibility condition]\label{def:admissibility} We say that \emph{the admissibility condition holds for a QP $(Q,S)$} if
\begin{enumerate}\item $(Q,S)$ is non-degenerate;
\item $S\in R\langle Q\rangle$, that is, $S$ is a finite linear combination of cyclic paths on $Q$;
\item the Jacobian algebra $\jacobqs$ is finite-dimensional; and
\item the $R$-algebra homomorphism $\phi_{(Q,S)}:R\langle Q\rangle/J_0(S)\rightarrow \jacobqs$ induced by the inclusion $R\langle Q\rangle\hookrightarrow R\langle\langle Q\rangle\rangle$ is an isomorphism.
\end{enumerate}
\end{defi}

\begin{remark}\begin{enumerate}\item If the admissibility condition holds for $(Q,S)$, then $R\langle Q\rangle/J_0(S)$ is a finite-dimensional $R\langle\langle Q\rangle\rangle$-module, and is therefore nilpotent. Thus, $J_0(S)$ contains a power of the ideal of $R\langle Q\rangle$ generated by the arrows of $Q$. On the other hand, non-degeneracy implies that $Q$ is 2-acyclic and hence all cycles appearing in $S$ have length at least three. Thus, $J_0(S)$ is contained in the square of the ideal of $R\langle Q\rangle$ generated by the arrows of $Q$. In other words, if the admissibility condition holds for $(Q,S)$, then $J_0(S)$ is an \emph{admissible ideal} of $R\langle Q\rangle$.
\item It is not true that conditions (1), (2) and (3) in Definition \ref{def:admissibility} imply condition (4): In \cite[Example~35]{Lqps} and \cite[Example~8.2]{Lreps} it is shown that the QP $\qstau$ associated to an ideal triangulation $\tau$ of a once-punctured torus satisfies (1), (2) and (3), but since the quotient $R\langle Q(\tau)\rangle/J_0(S(\tau))$ is infinite-dimensional, it does not satisfy (4).
\end{enumerate}
\end{remark}

Suppose the admissibility condition holds for $(Q,S)$. It is natural to ask whether the admissibility condition holds after applying a QP-mutation $\mu_i$ to $(Q,S)$. The following proposition says that the answer is yes as long as we can find a suitable right-equivalence when we apply the reduction process to the premutation $\premuti(Q,S)$.

\begin{prop}\label{prop:admissibility-of-mutation} Let $(Q,S)$ be a QP for which the admissibility condition holds, and fix $\arc\in Q_0$. Suppose $W$ is a finite reduced potential on $\mu_{\arc}(Q)$ and $(C,T)$ is a trivial QP, such that there exists a right-equivalence $\varphi:\widetilde{\mu}_{\arc}(Q,S)\longrightarrow(\mu_{\arc}(Q),W)\oplus(C,T)$ that restricts to an isomorphism between the path algebras $R\langle\widetilde{\mu}_{\arc}(Q)\rangle$ and $R\langle\mu_{\arc}(Q)\oplus C\rangle$. Then the admissibility condition holds for $(\mu_{\arc}(Q),W)$.
\end{prop}

\begin{proof} The QP $(\mu_{\arc}(Q),W)$ certainly lies in the right-equivalence class of the $\arc^{\operatorname{th}}$ mutation of $(Q,S)$. Therefore we only need to prove that $\phi_{(\mu_{\arc}(Q),W)}:R\langle \mu_{\arc}(Q)\rangle/J_0(W)\rightarrow \mathcal{P}(\mu_{\arc}(Q),W)$ is an isomorphism.

Since $(Q,S)$ satisfies the admissibility condition, there exists a positive integer $r$ such that every path on $Q$ of length greater than $r$ belongs to the ideal $J_0(S)$ of $R\langle Q\rangle$ generated by the cyclic derivatives of $S$. We claim that
\begin{equation}\label{eq:bound-for-length-inR<premut(qtau)>}\text{every path on $\premuti(Q)$ of length greater than $2r+9$ belongs to $J_0(\premuti(S))\subseteq R\langle\premuti(Q)\rangle$.}
\end{equation}
%
To prove \eqref{eq:bound-for-length-inR<premut(qtau)>} it suffices to show that if $a_1\ldots a_\ell$ is a path on $\premuti(Q)$ of length $\ell>2r+7$ that does not start nor end at $\arc$, then $a_1\ldots a_\ell\in J_0(\premuti(S))$. Since $\premuti(Q)$ does not have 2-cycles incident at $\arc$, we see that for any two $i$-hooks of $\premuti(Q)$ appearing in $a_1\ldots a_\ell$ there is at least an arrow or $i$-hook of $\qtau$ appearing in $a_1\ldots a_\ell$ between the given $i$-hooks of $\premuti(\qtau)$. From this fact and the identity
$$
\alpha^*\beta^*=\partial_{[\beta\alpha]}(\premuti(S))-\partial_{[\beta\alpha]}([S]),
$$
which holds for every $\arc$-hook $\beta\alpha$ of $Q$,
we deduce that $a_1\ldots a_\ell$ is congruent, modulo $J_0(\premuti(S))$, to a finite linear combination of paths on $\premuti(Q)$ that have length at least $\frac{\ell-1}{2}>r+3$ and do not involve any of the arrows of $\premuti(Q)$ incident to $\arc$. This means that we can suppose, without loss of generality, that $a_1\ldots a_\ell$ is a path on $\premuti(Q)$ of length $\ell>r+3$ and does not have any of the arrows of $\premuti(Q)$ incident to $\arc$ as a factor. Under this assumption, we see that $a_1\ldots a_\ell$ gives rise to a path $b_1\ldots b_l$ of $Q$ of length $l\geq \ell$ such that $[b_1\ldots b_l]=a_1\ldots a_\ell$. The paths $b_1\ldots b_{l-1}$ and $b_2\ldots b_l$  belong to $J_0(S)$. From this and Equations (6.6), (6.7) and (6.8) of \cite{DWZ1}, we deduce that $a_1\ldots a_\ell=[b_1\ldots b_l]$ belongs to $J_0(\premuti(S))$. This proves our claim \eqref{eq:bound-for-length-inR<premut(qtau)>}.

Note that $\premuti(Q)=\mu_{\arc}(Q)\oplus C$ (see the fourth item in Definition \ref{def:QP-stuff}). By \cite[Lemma~3.9]{DWZ1} (the cyclic Leibniz rule) and a slight modification of the proof of \cite[Proposition~4.5]{DWZ1}, we have $R\langle\premuti(Q)\rangle=R\langle\mu_i(Q)\oplus C\rangle=R\langle \mu_i(Q)\rangle\oplus L$, where $L$ is the two-sided ideal of $R\langle\premuti(Q)\rangle$ generated by the arrows in $C$, and
\begin{equation}\label{eq:varphi-preserves-J0}
\varphi(J_0(\premuti(S)))=J_0(W+T)=J_0(W)\oplus L,
\end{equation}
where $J_0(W+T)$ is taken inside $R\langle\premuti(Q)\rangle$ and $J_0(W)$ is taken inside $R\langle\mu_{\arc}(Q)\rangle$. Now, if $c_1\ldots c_\ell$ is a path on $\mu_{\arc}(Q)$ of length $\ell\geq 2r+9$, then $\varphi^{-1}(c_1\ldots c_\ell)$ is a finite linear combination of paths on $\premuti(Q)$ that have length at least $\ell$ and hence $\varphi^{-1}(c_1\ldots c_\ell)\in J_0(\premuti(S))$ by \eqref{eq:bound-for-length-inR<premut(qtau)>}. This implies, by \eqref{eq:varphi-preserves-J0}, that $c_1\ldots c_\ell\in J_0(W)$.

Since the Jacobian algebra $\mathcal{P}(\mu_{\arc}(Q),W)$ is finite-dimensional, it is a nilpotent $\mathcal{P}(\mu_{\arc}(Q),W)$-module, and this implies the surjectivity of  the $R$-algebra homomorphism $\phi_{(\mu_{\arc}(Q),W)}:R\langle\mu_{\arc}(Q)\rangle/J_0(W)\rightarrow\mathcal{P}(\mu_{\arc}(Q),W)$.

Let us prove that $\phi_{(\mu_{\arc}(Q),W)}$ is injective as well. Let $u\in R\langle\mu_{\arc}(Q)\rangle\cap J(W)$, so that $u$ is the limit of a sequence $(u_n)_{n>0}$ of elements of $R\langle\langle\mu_{\arc}(Q)\rangle\rangle$ that belong to the two-sided ideal generated by the cyclic derivatives of $W$. Let $\ell$ be an integer greater than $2r+9$ and the lengths of the paths appearing in the expression of $u$ as a finite linear combination of paths on $\mu_{\arc}(Q)$. By \eqref{eq:convergence-in-R<<Q>>}, there is an $n>0$ such that $u_n-u$ belongs to the $\ell^{\operatorname{th}}$ power of the ideal $\mathfrak{m}$ of $R\langle\langle\mu_{\arc}(Q)\rangle\rangle$ generated by the arrows of $\mu_{\arc}(Q)$. Furthermore, we can write $u_n$ as a finite sum of products of the form $x\partial_{\gamma}(W)y$, with $x,y\in R\langle\langle\mu_{\arc}(Q)\rangle\rangle$ and $\gamma$ some arrow of $\mu_{\arc}(Q)$. That is,
$$
u_n=\sum_{\gamma}\sum_{t,s}x_{t,\gamma}\partial_\gamma(W)y_{s,\gamma}.
$$
Let $x'_{t,{\gamma}}$ (resp. $y'_{s,\gamma}$) be the element of $R\langle\mu_{\arc}(Q)\rangle$ obtained from $x_{t,\gamma}$ (resp. $y_{s,\gamma}$) by deleting the summands that are $\field$-multiples of cycles of length greater than $\ell$. Then
$$
u'_n=\sum_{\gamma}\sum_{t,s}x'_{t,\gamma}\partial_\gamma(W)y'_{s,\gamma}
$$
is an element of $J_0(W)$ such that $u'_n-u_n\in\mathfrak{m}^\ell$. Thus $u'_n-u\in\mathfrak{m}^\ell$. From this we deduce that $u'_n-u\in J_0(W)$, since both $u'_n$ and $u$ belong to $R\langle\mu_{\arc}(Q)\rangle$ and all paths of length $\ell$ belong to $J_0(W)$. Consequently, $u\in J_0(W)$. We conclude that $\phi_{(\mu_{\arc}(Q),W)}$ is indeed injective.
\end{proof}

\begin{prop}\label{prop:reduction-2steps} Let $Q$ be a loop-free quiver, and $S$ a finite potential on $Q$. Suppose that $a_1,b_1,\ldots,a_N,b_N\in Q_1$ are $2N$ distinct arrows of $Q$ such that each product $a_kb_k$ is a 2-cycle and that the degree-2 component of $S$ is $S^{(2)}=\sum_kx_ka_kb_k$ for some non-zero scalars $x_1,\ldots,x_N$. Suppose further that
\begin{equation}\label{eq:S-normal-form}
S=\sum_kx_ka_kb_k+a_ku_k+v_kb_k+S',
\end{equation}
with $\deg_a(v_k)=\deg_b(u_k)=\deg_b(v_k)=0$, and $S'$ a potential not involving any of the arrows $a_1,b_1,\ldots,$ $a_N,b_N$, where, for a nonzero $u\in R\langle Q\rangle$, $\deg_a(u)$ is the maximum integer $d$ such that there is a non-zero summand of $u$ that has $d$ appearances of elements from $\{a_1,\ldots,a_N\}$, whereas $\deg_a(0)=0$ ($\deg_b(u)$ is defined similarly). Then the reduced part $(Q_{\operatorname{red}},S_{\operatorname{red}})$, the trivial part $(Q_{\operatorname{triv}},S_{\operatorname{triv}})$ and the right-equivalence $\varphi:(Q,S)\rightarrow(Q_{\operatorname{red}},S_{\operatorname{red}})\oplus(Q_{\operatorname{triv}},S_{\operatorname{triv}})$ can be chosen in such a way that
\begin{itemize}\item $S_{\operatorname{red}}$ is a finite potential, and
\item $\varphi$ maps $R\langle Q\rangle$ onto $R\langle Q_{\operatorname{red}}\oplus Q_{\operatorname{triv}}\rangle$.
\end{itemize}
\end{prop}

\begin{proof} 
Let $\varphi:R\langle\langle Q\rangle\rangle\rightarrow R\langle\langle Q\rangle\rangle$ be the $R$-algebra isomorphism given by
$$
a_k\mapsto a_k-x_k^{-1}v_k, \ \ \ b_k\mapsto b_k-x_k^{-1}u_k,\ \ \ \text{for $k=1,\ldots,N$},
$$
and the identity on the rest of the arrows of $Q$. Since $u_k,v_k\in R\langle Q\rangle$, $\varphi$ certainly maps $R\langle Q\rangle$ into itself. It is clear that $\varphi^{-1}$ also does if $\deg_a(u_k)=0$ for all $k=1,\ldots,N$. So, suppose that $\max\{\deg_a(u_k)\suchthat k=1,\ldots,N\}>0$. Note that since $\deg_{a}(v_k)=\deg_b(u_k)=\deg_b(v_k)=0$ for every $k=1,\ldots,N$, we can recursively define elements $u_{k,0},\ldots,u_{k,\deg_a(u_k)}$, with the following properties
\begin{itemize}
\item $u_{k,0}=x_k^{-1}u_{k}$;
\item $\varphi(u_{k,\ell})=u_{k,\ell}+u_{k,\ell+1}$ for $\ell=0,\ldots,\deg_a(u_k)-1$, whereas $\varphi(u_{k,\deg_a(u_k)})=u_{k,\deg_a(u_k)}$;
\item $\deg_a(u_{k,\ell})=\deg_a(u_k)-\ell$ and $\deg_b(u_{k,\ell})=0$ for $\ell=0,\ldots,\deg_a(u_k)$.
\end{itemize}
But then, $\varphi^{-1}$ is given by
$$
a_k\mapsto a_k+x_k^{-1}v_k, \ \ \ b_k\mapsto b_k+\sum_{\ell=0}^{\deg_a(u_k)}(-1)^{\ell}u_{k,\ell}, \ \ \ \text{for $k=1,\ldots,N$},
$$
and the identity on the rest of the arrows of $Q$. This shows that $\varphi^{-1}$ maps $R\langle Q\rangle$ into itself. Therefore, $\varphi$ maps $R\langle Q\rangle$ bijectively onto itself. Now, $\varphi(S)=
\sum_kx_ka_kb_k+x_ka_ku_{k,1}-v_ku_{k,1}-x_k^{-1}v_ku_k+S'$
is cyclically equivalent to a potential of the form $\sum_k\left(x_ka_kb_k+a_ku'_k\right)+S''$,
where $S''$ is a finite potential not involving any of the arrows $a_1,b_1,\ldots,a_N,b_N$, and $\max\{\deg_a(u_k')\suchthat k=1,\ldots,N\}<\max\{\deg_a(u_k)\suchthat k=1,\ldots,N\}$. Thus, the proposition follows by induction on $n=\max\{\deg_a(u_k)\suchthat k=1,\ldots,N\}$.
\end{proof}

\begin{ex} Consider Figure \ref{Fig:admissibility},
        \begin{figure}[!h]
                \caption{}\label{Fig:admissibility}
                \centering
                \includegraphics[scale=.5]{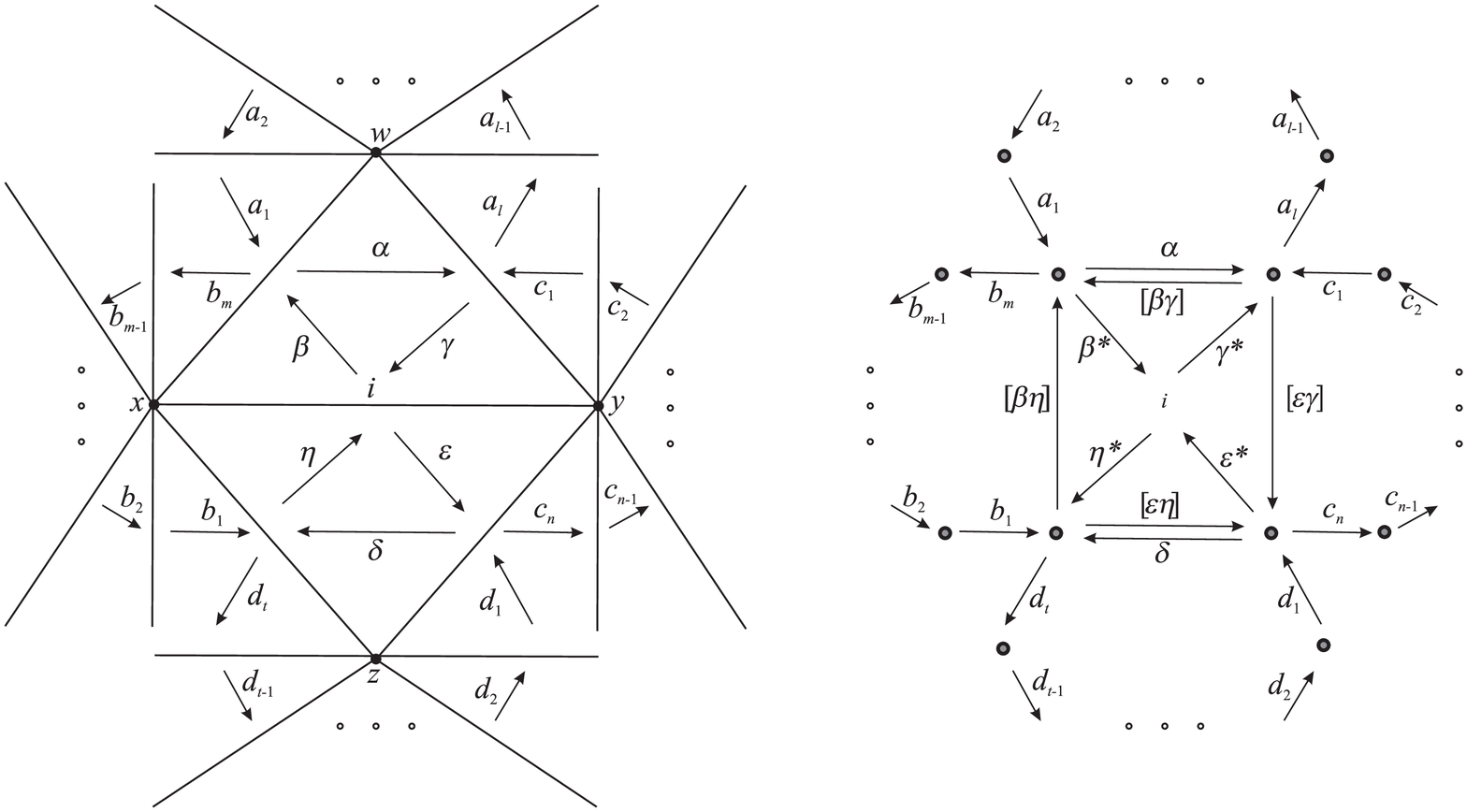}
        \end{figure}
where a small piece of an ideal triangulation $\tau$ is drawn on the left, and the quiver $\widetilde{\mu}_i(Q(\tau))$ is sketched on the right. The potential $\widetilde{\mu}_i(S(\tau))$, written in the form \eqref{eq:S-normal-form}, is given by
$$
\widetilde{\mu}_i(S(\tau))=\alpha[\beta\gamma]+\alpha (wa_1\ldots a_l)+\gamma^*\beta^*[\beta\gamma]
+\delta[\varepsilon\eta]+\delta (zd_1\ldots d_t)+\eta^*\varepsilon^*[\varepsilon\eta]
+S',
$$
where $S'=x[\beta\eta] b_1\ldots b_m+y[\varepsilon\gamma] c_1\ldots c_n
+\eta^*\beta^*[\beta\eta]+\gamma^*\varepsilon^*[\varepsilon\gamma]+S''(\tau)$, with $S''(\tau)$ a potential not involving any of the arrows $\alpha,\beta^*,\gamma^*,\delta,[\beta\gamma],[\varepsilon\eta],[\beta\eta],[\varepsilon\gamma]$. Setting $u_\alpha=wa_1\ldots a_l$, $u_\delta=zd_1\ldots d_t$, $v_{[\beta\gamma]}=\gamma^*\beta^*$, $v_{[\varepsilon\eta]}=\eta^*\varepsilon^*$, we see that, in the notation of Proposition \ref{prop:reduction-2steps}, we have
$\deg_a(v_{[\beta\gamma]})=\deg_a(v_{[\varepsilon\eta]})=
\deg_{b}(u_\alpha)=\deg_b(u_\delta)=\deg_{b}(v_{[\beta\gamma]})=\deg_b(v_{[\varepsilon\eta]})=0$.
Proposition \ref{prop:reduction-2steps} then guarantees the involvement of only finitely many cycles in $\mu_i(S(\tau))$, and the existence of a right-equivalence
$\widetilde{\mu}_i(Q(\tau),S(\tau))\to
\mu_i(Q(\tau),S(\tau))\oplus(\widetilde{\mu}_i(Q(\tau))_{\operatorname{triv}},\widetilde{\mu}_i(S(\tau))_{\operatorname{triv}})$ that maps the (incomplete) path algebra $R\langle\widetilde{\mu}_i(Q(\tau))\rangle$ bijectively onto the (incomplete) path algebra $R\langle\mu_i(Q(\tau))\oplus \widetilde{\mu}_i(Q)_{\operatorname{triv}}\rangle$. Applying Proposition \ref{prop:admissibility-of-mutation}, we see that if we manage to show that $(Q(\tau),S(\tau))$ satisfies the admissibility condition, then $\mu_i(Q(\tau),S(\tau))$ will automatically satisfy the admissibility condition as well.
\end{ex}

\begin{ex} Let us give an example involving a more complicated triangulation. Consider Figure \ref{Fig:admissibility_ugly},
        \begin{figure}[!h]
                \caption{}\label{Fig:admissibility_ugly}
                \centering
                \includegraphics[scale=.4]{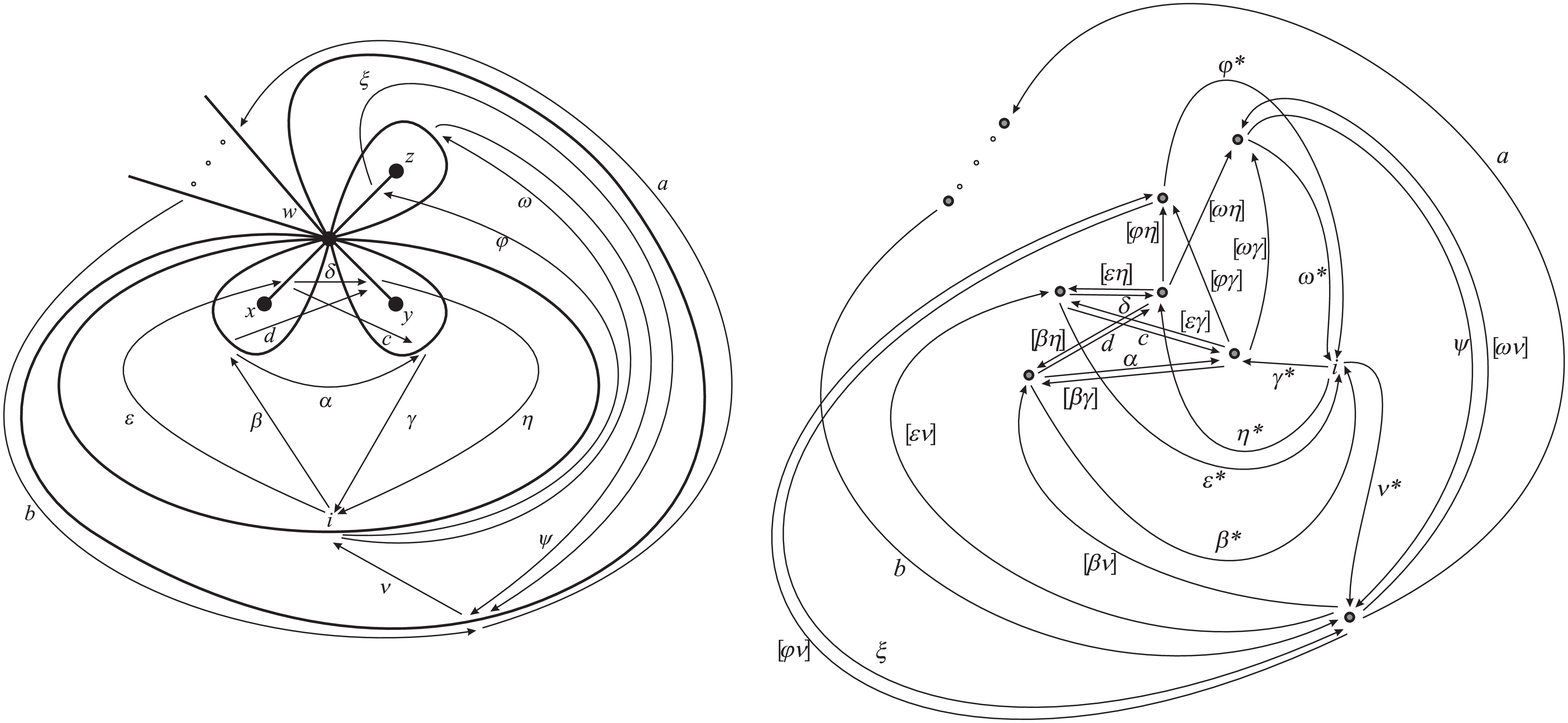}
        \end{figure}
where a small piece of an ideal triangulation $\tau$ is drawn on the left, and the quiver $\widetilde{\mu}_i(Q(\tau))$ is sketched on the right. The potential $\widetilde{\mu}_i(S(\tau))$, written in the form \eqref{eq:S-normal-form}, is given by
\begin{eqnarray}\nonumber
\widetilde{\mu}_i(S(\tau)) & = &
\alpha[\beta\gamma]+\gamma^*\beta^*[\beta\gamma]
+\psi[\omega\nu]+\nu^*\omega^*[\omega\nu]
+\frac{\delta[\varepsilon\eta]}{xy}+\delta(w[\varepsilon\nu] b\ldots a\xi[\varphi\eta])+\eta^*\varepsilon^*[\varepsilon\eta]\\
\nonumber
 & &
-\frac{c[\varepsilon\gamma]}{x}+\gamma^*\varepsilon^*[\varepsilon\gamma]
-\frac{d[\beta\eta]}{y}+\eta^*\beta^*[\beta\eta]
-\frac{\xi[\varphi\nu]}{z}+\nu^*\varphi^*[\varphi\nu]+S',
\end{eqnarray}
where $S'=\nu^*\beta^*[\beta\nu]
+\gamma^*\omega^*[\omega\gamma]
+\eta^*\omega^*[\omega\eta]
+\nu^*\varepsilon^*[\varepsilon\nu]
+\gamma^*\varphi^*[\varphi\gamma]
+\eta^*\varphi^*[\varphi\eta]
+S''(\tau)$,
with $S''(\tau)$ a potential not involving any of the arrows
$\alpha,\psi,\delta,\xi,c,d,
\gamma^*,\beta^*,\eta^*,\nu^*,\omega^*,\varepsilon^*,
[\beta\gamma],[\beta\eta],$ $[\beta\nu],[\omega\gamma],[\omega\eta],[\omega\nu],[\varepsilon\gamma],[\varepsilon\eta],[\varepsilon\nu],$
$\varphi^*,[\varphi\gamma],[\varphi\eta],[\varphi\nu]$. Just as in the previous example, we can easily check that Proposition \ref{prop:reduction-2steps} can be applied in combination with Proposition \ref{prop:admissibility-of-mutation} to deduce that if we manage to show that $(Q(\tau),S(\tau))$ satisfies the admissibility condition, then $\mu_i(Q(\tau),S(\tau))$ will automatically satisfy the admissibility condition as well.
\end{ex}

The following is the main result of this section.

\begin{thm}\label{thm:admissibility-holds-for-Qtau} The admissibility condition holds for every QP of the form $\qstau$ with $\tau$ a tagged triangulation of a surface with non-empty boundary.
\end{thm}

\begin{remark} In the particular case when $\tau$ is an ideal triangulation of a surface with non-empty boundary, this had already been realized and stated by the second author in \cite{Lqps}. However, in \cite{Lqps} it is only shown that there exists an ideal triangulation whose QP satisfies the admissibility condition, and the proof that this condition holds for all QPs of ideal triangulations is omitted. Since admissibility turns out to be a delicate point in certain situations (e.g. in some approaches to the study of Donaldson-Thomas invariants, see \cite{Nagao} and the Comments at the end of its Introduction), in Proposition \ref{prop:Qtauadmissible=>Qsigmaadmissible} below we will present the proof that was omitted in \cite{Lqps}.
\end{remark}

Theorem \ref{thm:admissibility-holds-for-Qtau} will be a consequence of the following Proposition.

\begin{prop}\label{prop:Qtauadmissible=>Qsigmaadmissible} Suppose $\tau$ and $\sigma$ are ideal triangulations of $\surf$ related by a flip. If the admissibility condition holds for $\qstau$, then it holds for $\qssigma$ as well.
\end{prop}

\begin{proof} Let $\arc\in\tau$ be such that $\sigma=f_{\arc}(\tau)$. By Definition \ref{def:QP-of-tagged-triangulation}, Theorem \ref{thm:flip<->mutation-ideal-triangs}, and Proposition \ref{prop:admissibility-of-mutation}, we only have to show the existence of a trivial QP $(C,T)$ and a right-equivalence $\varphi:\premuti\qstau\longrightarrow\qssigma\oplus(C,T)$ that restricts to an isomorphism between the path algebras $R\langle\premuti(Q(\tau))\rangle$ and $R\langle\qsigma\oplus C\rangle$.

Up to cyclical equivalence, we can suppose that the degree-2 component of $\premuti(\stau)$ is a sum of elements of the form $x\alpha[\beta\gamma]$, where $\alpha\beta\gamma$ is an oriented 3-cycle of $\qtau$ such that $\beta\gamma$ is an $\arc$-hook (in $\qtau$), and $x$ is a non-zero element of the ground field $\field$. Then, by the second item in Remark \ref{rem:hooks-only-once}, the premutation $(\premuti(\qtau),\premuti(\stau))$ satisfies the hypothesis of Proposition \ref{prop:reduction-2steps} if we set the $a$-arrows (resp. the $b$-arrows) to be the arrows $\alpha$ (resp. the arrows $[\beta\gamma]$) from the terms $x\alpha[\beta\gamma]$. Hence there exists a right-equivalence $\psi_1$ from $(\premuti(\qtau),\premuti(\stau))$ to the direct sum $(\premuti(\qtau)_{\operatorname{red}},\premuti(\stau)_{\operatorname{red}})
\oplus(\premuti(\qtau)_{\operatorname{triv}},\premuti(\stau)_{\operatorname{triv}})$ that maps $R\langle\premuti(\qtau)\rangle$ onto $R\langle\premuti(\qtau)_{\operatorname{red}}\oplus\premuti(\qtau)_{\operatorname{triv}}\rangle$, with $\premuti(\stau)_{\operatorname{red}}$ a finite potential.

Recall that a \emph{change of arrows} is an isomorphism $\psi$ of complete path algebras such that, when we write $\psi|_A=(\psi^{(1)},\psi^{(2)})$ as in \eqref{eq:bimodule-homomorphisms}, we have $\psi^{(2)}=0$. It is clear that changes of arrows map path algebras onto path algebras. As seen in the proof of \cite[Theorem 30]{Lqps}, there is a change of arrows $\overline{\psi}_2:R\langle\langle\premuti(\qtau)_{\operatorname{red}}\rangle\rangle\rightarrow R\langle\langle\qsigma\rangle\rangle$ that serves as a right-equivalence between $\mu_{\arc}\qstau=(\premuti(\qtau)_{\operatorname{red}},\premuti(\stau)_{\operatorname{red}})$ and $\qssigma$. The proposition follows by setting $(C,T)=(\premuti(\qtau)_{\operatorname{triv}},\premuti(\stau)_{\operatorname{triv}})$ and taking $\varphi:R\langle\langle\premuti(\qtau)\rangle\rangle\rightarrow R\langle\langle\qsigma\oplus C\rangle\rangle$ to be the composition
$$
R\langle\langle\premuti(\qtau)\rangle\rangle\overset{\psi_1}{\longrightarrow} R\langle\langle\premuti(\qtau)_{\operatorname{red}}\oplus C\rangle\rangle\overset{\psi_2}{\longrightarrow} R\langle\langle Q(\sigma)\oplus C\rangle\rangle,
$$
where $\psi_2$ acts as $\overline{\psi}_2$ on $\qsigma$ and as the identity on $C$.
\end{proof}

\begin{proof}[Proof of Theorem \ref{thm:admissibility-holds-for-Qtau}] By Theorems \ref{thm:flip<->mutation-tagged-triangs} and \ref{thm:flip<->mutation-ideal-triangs}, $\qstau$ is non-degenerate. The potential $\stau$ is obviously a finite linear combination of cyclic paths on $\qtau$. Since finite-dimensionality of Jacobian algebras is a QP-mutation invariant, the fact that $\jacobqstau$ is finite-dimensional follows from Theorems \ref{thm:flip<->mutation-tagged-triangs} and \ref{thm:flip<->mutation-ideal-triangs}.

In the proof of Theorem 36 of \cite{Lqps}, the existence is shown of an ideal triangulation $\tau'$ of $\surf$, without self-folded triangles, such that $\phi_{(Q(\tau'),S(\tau'))}$ is an isomorphism. This implies, in view of Definition \ref{def:QP-of-tagged-triangulation}, that for every $\epsilon:\punct\rightarrow\{-1,1\}$, the set $\bar{\Omega}'_\epsilon$ contains a tagged triangulation for whose QP the admissibility condition holds. From this fact, Definition \ref{def:QP-of-tagged-triangulation} and Proposition \ref{prop:Qtauadmissible=>Qsigmaadmissible}, we deduce that $\phi_{(Q(\tau),S(\tau))}$ is an isomorphism. Therefore, the admissibility condition holds for $\qstau$.
\end{proof}

\section{Cluster monomials, proper Laurent monomials and positivity}\label{sec:monomials}

\begin{defi} Let $B$ be a skew-symmetric matrix and $\widetilde{B}$ be an $(n+r)\times n$ integer matrix whose top $n\times n$ submatrix is $B$.
\begin{enumerate}\item Denote by $\myAA(\widetilde{B})$ the cluster algebra of the cluster pattern that has $\widetilde{B}$ at its initial seed.
\item Let $\x=(x_1,\ldots,x_n)$ be a cluster in $\myAA(\widetilde{B})$. A \emph{proper Laurent monomial in $\x$} is a product of the form $x_1^{c_1}\ldots x_n^{c_n}$ with at least one negative exponent $c_l$.
\item We say that $\myAA(\widetilde{B})$ \emph{has the proper Laurent monomial property} if for any two different clusters $\x$ and $\x'$ of $\myAA(\widetilde{B})$, every monomial in $\x'$ in which at least one element from $\x'\setminus\x$ appears with positive exponent is a $\Z\semifield$-linear combination of proper Laurent monomials in $\x$.
\end{enumerate}
\end{defi}

\begin{thm}\label{thm:admissibility=>proper-Laurent-monomials} Let $B$ and $\widetilde{B}$ be as above and $S$ be a non-degenerate potential on the quiver $Q=Q(B)$. Put $\widetilde{B}$ at the initial seed of a cluster pattern on the $n$-regular
tree $\TT$ with initial vertex $t_0$. Suppose that there exists a connected subgraph $G$ of $\TT$ that contains $t_0$ and has the following properties:
\begin{itemize}
\item Every (unordered) cluster of $\myAA(\widetilde{B})$ appears in at least one of the seeds attached to the vertices of $G$ by the cluster pattern under consideration;
\item for every path $t_0\frac{\phantom{---}}{i_1} t_1\frac{\phantom{---}}{i_2}\ldots\frac{\phantom{---}}{i_{m-1}}t_{m-1}\frac{\phantom{---}}{i_{m}}t$ contained in $G$, the QP obtained by applying to $(Q,S)$ the QP-mutation sequence $\mu_{i_1},\mu_{i_2},\ldots,\mu_{i_m}$, satisfies the admissibility condition.
\end{itemize}
Then $\myAA(\widetilde{B})$ has the proper Laurent monomial property.
\end{thm}

\begin{remark} \begin{itemize}\item QP-mutations are defined up to right-equivalence, so what the second property in the statement of Theorem \ref{thm:admissibility=>proper-Laurent-monomials} means is that there is at least one QP in the corresponding right-equivalence class for which the admissibility condition holds.
\item The graph $G$ in the statement of Theorem \ref{thm:admissibility=>proper-Laurent-monomials} may not contain all the vertices of $\TT$: in principle, it is possible that a given cluster appears in different vertices of $\TT$, but that in some of these vertices the accompanying QP (and any representative of its right-equivalence class) does not satisfy the admissibility condition. We only require that in at least one of these vertices the accompanying QP satisfies the admissibility condition, in such a way that the resulting collection of vertices of $\TT$ forms a connected subgraph. Thus, for example, in a cluster algebra of type $\mathbb{A}_2$, checking that the admissibility condition holds at (the QP attached to) every vertex of a 5-vertex connected subgraph of $\mathbb{T}_2$ would suffice in order to apply Theorem \ref{thm:admissibility=>proper-Laurent-monomials} and deduce the proper Laurent monomial property.
\end{itemize}
\end{remark}

\begin{proof}[Proof of Theorem \ref{thm:admissibility=>proper-Laurent-monomials}]
Let $\x'=(x'_1,\ldots,x'_n)$ and $\x''=(x''_1,\ldots,x''_n)$ be two different clusters of $\myAA(\widetilde{B})$. Then there exist vertices $t$ and $s$ of $G$,  paths $t_0\frac{\phantom{---}}{i_1} t_1\frac{\phantom{---}}{i_2}\ldots\frac{\phantom{---}}{i_{m-1}}t_{m-1}\frac{\phantom{---}}{i_{m}}t$ and $t_0\frac{\phantom{---}}{j_1} s_1\frac{\phantom{---}}{j_2}\ldots\frac{\phantom{---}}{i_{\ell-1}}j_{\ell-1}\frac{\phantom{---}}{j_{\ell}}s$ on $G$, and $(n+r)\times n$ integer matrices $\widetilde{B}'$ and $\widetilde{B}''$ such that $(\widetilde{B}',\x')$ (resp $(\widetilde{B}'',\x'')$) is the unique seed attached to the vertex $t$ (resp. $s$) by the cluster pattern under consideration. For $j=1,\ldots,n$, denote by $\mathbf{b}'_j$ the $j^{th}$ column of $\widetilde{B}'$. According to Theorem \ref{thm:FZ-separation-of-variables}, for $k=1,\ldots,n$ we have
\begin{equation}\label{eq:separation-formula}
x''_{k}=Y_kF_{k;s}^{B';t}|_{\myFF}(\x'^{\mathbf{b}'_1},\ldots,\x'^{\mathbf{b}'_n})x_1'^{g_{1,k}}\ldots x_n'^{g_{n,k}},
\end{equation}
where $Y_k$ is some element of $\semifield$ and we are using the shorthand
\begin{equation}
\x'^{\mathbf{b}'_l}=\prod_{k=1}^{n+r}x_k'^{b'_{kl}} \ \text{for each column} \ \mathbf{b}'_l=\left[\begin{array}{c}b'_{1l}\\ \vdots\\ b'_{(n+r)l}\end{array}\right] \ \text{of} \ \widetilde{B}'.
\end{equation}
Moreover, the $F$-polynomial $F_{k;s}^{B';t}$ and the $\mathbf{g}$-vector $\mathbf{g}_{k;s}^{B',t}=[g_{1,k},\ldots, g_{n,k}]^{\operatorname{t}}$ are defined in terms of the cluster pattern that uses $t$ as initial vertex, has principal coefficients at that vertex and has $B'$ as initial exchange matrix.

Now, let $Q'=Q(B')$ and $Q''=Q(B'')$ be the 2-acyclic quivers respectively associated to the skew-symmetric matrices $B'$ and $B''$. Furthermore, let $S'$ (resp. $S''$) be the potential on $Q'$ (resp. $Q''$) obtained by applying to $(Q,S)$ the QP-mutation sequence $\mu_{i_1},\ldots,\mu_{i_m}$ (resp. $\mu_{j_1},\ldots,\mu_{j_\ell}$). For each $k=1,\ldots,n$ define
\begin{equation}
\mathcal{M}_k=(M(k),V(k))=\mu_{i_m}\ldots\mu_{i_1}\mu_{j_1}\ldots\mu_{j_\ell}\left(\mathcal{S}^-_{k}(Q'',S'')\right),
\end{equation}
which is a decorated representation of $(Q',S')$.
By Theorem \ref{thm:DWZ-gvectors-Fpoynomials} and \eqref{eq:separation-formula} above,
\begin{equation}\label{eq:xk=Fpolynomial-times-gvector}
x''_{k}=Y_kF_{\mathcal{M}_k}(\x'^{\mathbf{b}'_1},\ldots,\x'^{\mathbf{b}'_n})x_1'^{g^{\mathcal{M}_k}_{1,k}}\ldots x_n'^{g^{\mathcal{M}_k}_{n,k}}
\end{equation}
(even if $j_{\ell}\neq k$), where $\mathbf{g}_{\mathcal{M}_k}=[g^{\mathcal{M}_k}_{1,k},\ldots,g^{\mathcal{M}_k}_{n,k}]^{\operatorname{t}}$ is the $\operatorname{g}$-vector of the decorated representation $\mathcal{M}_k=(M(k),V(k))$ and
\begin{equation}
F_{\mathcal{M}_k}=\sum_{\mathbf{e}\in\mathbb{N}^n}\chi(\Gr_{\mathbf{e}}(M(k)))\operatorname{X}^{\mathbf{e}}
\end{equation}
is its $F$-polynomial. (Here, $\operatorname{X}$ is a tuple of $n$  indeterminates).

Let $\x''^{\mathbf{a}}=x_1''^{a_1}\ldots x_n''^{a_n}$ be a monomial in $\x''$ in which at least one element from $\x''\setminus\x'$ appears with positive exponent, and define
\begin{equation}
\mathcal{M}=(M,V)=\mathcal{M}_1^{a_1}\oplus\ldots\oplus\mathcal{M}_n^{a_n}.
\end{equation}
Then, by \eqref{eq:xk=Fpolynomial-times-gvector} above, and Proposition 3.2 and Equation (5.1) of \cite{DWZ2},
\begin{equation}
\x''^{\mathbf{a}}=\left[\prod_{k=1}^nY_k^{a_k}\right]F_{\mathcal{M}}(\x'^{\mathbf{b}'_1},\ldots,\x'^{\mathbf{b}'_n})
x_1'^{g^{\mathcal{M}}_{1}}\ldots x_n'^{g^{\mathcal{M}}_{n}},
\end{equation}
where $\mathbf{g}_{\mathcal{M}}=[g^{\mathcal{M}}_1,\ldots,g^{\mathcal{M}}_n]^{\operatorname{t}}=
\sum_{k=1}^na_k\mathbf{g}_{\mathcal{M}_k}$ is the $\mathbf{g}$-vector of $\mathcal{M}$ and
\begin{equation}
F_{\mathcal{M}}=\sum_{\mathbf{e}\in\mathbb{N}^n}\chi(\Gr_{\mathbf{e}}(M))\operatorname{X}^{\mathbf{e}}
\end{equation}
is its $F$-polynomial. We therefore get
\begin{equation}\label{eq:cluster-monomial=Caldero-Chapoton}
\x''^{\mathbf{a}}=\left[\prod_{k=1}^nY_k^{a_k}\right]
\sum_{\mathbf{e}\in\mathbb{N}^n}\chi(\Gr_{\mathbf{e}}(M))\x'^{\widetilde{B}'\mathbf{e}}x_1'^{g^{\mathcal{M}}_{1}}\ldots x_n'^{g^{\mathcal{M}}_{n}}.
\end{equation}

Since QP-mutations preserve direct sums of QP-representations, we see that $\mathcal{M}$ is QP-mutation equivalent to a negative QP-representation. Since the $E$-invariant is a QP-mutation invariant and negative QP-representations have zero $E$-invariant, we deduce that $E(\mathcal{M})=0$. Suppose that $\mathcal{M}=(M,V)$ is not a positive QP-representation, and set $I=\{k\in [1,n]\suchthat \mathcal{M}_k$ is not positive$\}$. Then for each $k\in I$ there is an index $i_k\in[1,n]$ such that $\mathcal{M}_k=\mathcal{S}^-_{i_k}(Q',S')$ (this follows from the fact that $\mathcal{M}_k$ is an indecomposable QP-representation, and every indecomposable QP-representation is either positive or negative), and this implies $x''_k=x'_{i_k}$. Write $\z=\x''^{\mathbf{a}-\mathbf{c}}$, where
$$
c_k=\begin{cases}a_k & \text{if $k\in I$}\\
0 & \text{otherwise},
\end{cases}
$$
so that $\z$ is a cluster monomial in the cluster variables from $\x''$. Notice that $\mathbf{a}-\mathbf{c}\neq 0$, for at least one element of $\x''\setminus\x'$ appears in $\x''^{\mathbf{a}}$ with positive exponent. Therefore, the QP-representation
$$
\mathcal{M}_\z=\mathcal{M}_1^{a_1-c_1}\oplus\ldots\oplus\mathcal{M}_n^{a_n-c_n},
$$
is positive. Besides yielding an expression \eqref{eq:cluster-monomial=Caldero-Chapoton} for $\z$, $\mathcal{M}_\z$ is a direct summand of $\mathcal{M}$. Thus for all $k\in I$, the $i_k^{\operatorname{th}}$ vector space of the positive part of $\mathcal{M}_z$ is zero (indeed, since $M_{i_k}=0$, which follows from Corollary 8.3 of \cite{DWZ2}). This implies, by Corollary 5.5 of \cite{DWZ2}, that the $i_k^{\operatorname{th}}$ entry of the denominator vector of $\z$ with respect to the cluster $\x'$ is zero. This means that if we manage to prove that $\z$ is a $\Z\semifield$-linear combination of proper Laurent monomials in $\x'$, then we will have proved that $\x''^{\mathbf{a}}$ is also a $\Z\semifield$-linear combination of proper Laurent monomials in $\x'$. The advantage is that $\mathcal{M}_z$ is a positive QP-representation.

Justified by the previous paragraph we now assume, without loss of generality, that the QP-representation $\mathcal{M}$ associated to the cluster monomial $\x''^{\mathbf{a}}$ is positive. A quick look at \eqref{eq:cluster-monomial=Caldero-Chapoton} reveals us that in order to prove that $\x''^{\mathbf{a}}$ is a $\Z\semifield$-linear combination of proper Laurent monomials in $\x'$, it suffices to show that for every $\mathbf{e}\in\mathbb{N}^n$ for which $\chi(\Gr_{\mathbf{e}}(M))\neq 0$, the vector $B'\mathbf{e}+\mathbf{g}_\mathcal{M}$ has at least one negative entry.

We deal with non-zero $\mathbf{e}$ first. Suppose that there exists a non-zero subrepresentation $N$ of $M$ of dimension vector $\mathbf{e}$. Since $B'$ is skew-symmetric, the dot product $\mathbf{e}\cdot(B'\mathbf{e})$ is zero, and hence $\mathbf{e}\cdot(\mathbf{g}_{M}+B'\mathbf{e})=\mathbf{e}\cdot\mathbf{g}_M$. Thus, to prove that the vector $B'\mathbf{e}+\mathbf{g}_\mathcal{M}$ has at least one negative entry, it is enough to show that the number $\mathbf e\cdot\mathbf{g}_M$ is negative. Since $(Q',S')$ satisfies the admissibility condition, the $E$-invariant has a homological interpretation. More precisely, by \cite[Corollary~10.9]{DWZ2} we have
$$
E(M)=\mathbf{dim}(M)\cdot\mathbf g_M+\textrm{dim}\textrm{Hom}_{\mathcal{P}(Q',S')}(M,M)=0=\textrm{dim}\textrm{Hom}_{\mathcal{P}(Q',S')}(\tau^{-1}M,M)
$$
$$
\text{and} \ \ E^{inj}(N,M)=\mathbf e\cdot\mathbf g_M+\textrm{dim}\textrm{Hom}_{\mathcal{P}(Q',S')}(N,M)=\textrm{dim}\textrm{Hom}_{\mathcal{P}(Q',S')}(\tau^{-1}M,N),
$$
where $\tau$ is the Auslander-Reiten translation.

Since $N$ is a subrepresentation of $M$, there is an injection
$$
\textrm{Hom}_{\mathcal{P}(Q',S')}(\tau^{-1}M,N)\rightarrow\textrm{Hom}_{\mathcal{P}(Q',S')}(\tau^{-1}M,M)
$$
and so $E^{inj}(N,M)\leq E(M)=0$. It follows that $E^{inj}(N,M)=0$ and hence $\mathbf e\cdot\mathbf g_M=-\textrm{dim}\textrm{Hom}(N,M)<0$ as desired.

It remains to treat the case $\mathbf{e}=0$. Since $\mathbf{dim}(M)\cdot \mathbf{g}_{\mathcal{M}}=-\textrm{dim}\textrm{Hom}_{\mathcal{P}(Q',S')}(M,M)<0$, the vector $\mathbf{g}_{\mathcal{M}}$ has a negative entry as desired. Theorem \ref{thm:admissibility=>proper-Laurent-monomials} is proved.
\end{proof}


\begin{thm}\label{thm:positive-have-positive-coefficients} Let $B$ and $\widetilde{B}$ be as above. If $\myAA(\widetilde{B})$ has the proper Laurent monomial property, then any positive element of $\myAA(\widetilde{B})$ that belongs to the $\Z\semifield$-submodule of $\myAA(\widetilde{B})$ spanned by all cluster monomials is a $\Z_{\geq 0}\semifield$-linear combination of cluster monomials. Furthermore, cluster monomials are linearly independent over $\Z\semifield$.
\end{thm}

\begin{proof} Suppose $X$ is a positive element of $\myAA(\widetilde{B})$ that belongs to the $\Z\semifield$-submodule of $\myAA(\widetilde{B})$ spanned by all cluster monomials, so that we can write
\begin{equation}
X=\sum_{\x,\mathbf{a}}y_{\x,\mathbf{a}}\x^{\mathbf{a}}
\end{equation}
for some elements $y_{\x,\mathbf{a}}\in\Z\semifield$, all but a finite number of which are zero, where the sum runs over all cluster monomials of $\myAA(\widetilde{B})$ and all vectors $\mathbf{a}\in\mathbb{N}^n$. (As before, we have used the shorthand $\x^{\mathbf{a}}=x_1^{a_1}\ldots x_{n}^{a_n}$ for a cluster $\x=(x_1,\ldots,x_n)$ and a vector $\mathbf{a}=(a_1,\ldots,a_n)\in\mathbb{N}^n$).

Fix a cluster $\x$ and a vector $\mathbf{a}\in\mathbb{N}^n$. By Theorem \ref{thm:admissibility=>proper-Laurent-monomials}, each cluster monomial $\x'^{\mathbf{a}'}$ where at least one element of $\x'$ appears with positive exponent is a $\Z\semifield$-linear combination of proper Laurent monomials in $\x$. This means that in the expansion of $X$ as a Laurent polynomial in $\x$ with coefficients in $\Z\semifield$, the coefficient of the monomial $\x^{\mathbf{a}}$ is precisely $y_{\x,\mathbf{a}}$ (for the monomials in $\x$ are certainly linearly independent over $\Z\semifield$). Since $X$ is positive, this means that all integers appearing in $y_{\x,\mathbf{a}}$ are non-negative.

The proof that cluster monomials are linearly independent follows from an argument similar to the one above. Indeed, suppose
\begin{equation}\label{eq:cluster-mons-are-lin-ind}
\sum_{\x,\mathbf{a}}y_{\x,\mathbf{a}}\x^{\mathbf{a}}=0
\end{equation}
is an expression of $0$ as a $\Z\semifield$-linear combination of cluster monomials. Fix a cluster $\x$ and a vector $\mathbf{a}\in\mathbb{N}$. Every cluster monomial $\x'^{\mathbf{a}'}$ where at least one element of $\x'$ appears with positive exponent is a $\Z\semifield$-linear combination of proper cluster monomials in $\x$. Thus, when we express the left hand side of \eqref{eq:cluster-mons-are-lin-ind} as a Laurent polynomial in $\x$ with coefficients in $\Z\semifield$, the coefficient of $\x^{\mathbf{a}}$ will be precisely $y_{\x,\mathbf{a}}$. After clearing out denominators, the equality \eqref{eq:cluster-mons-are-lin-ind} yields $y_{\x,\mathbf{a}}=0$.

Theorem \ref{thm:positive-have-positive-coefficients} is proved.
\end{proof}

%
%

\begin{coro}\label{coro:Laurent-for-surfaces} Let $\surf$ be a surface with non-empty boundary and $\myAA\surf$ be a cluster algebra of geometric type associated to $\surf$. Then $\myAA\surf$ has the proper Laurent monomial property. In particular, cluster monomials of $\myAA\surf$ are linearly independent over the group ring of the ground semifield.
\end{coro}

\begin{proof} Call \emph{labeled triangulation} a tagged triangulation $\tau$ whose elements have been labeled with the numbers $1,\ldots,n$, (where $n$ is the number of elements of $\tau$), in such a way that different arcs receive different labels. Fix an ideal triangulation $\sigma$ of $\surf$ and attach it to an initial vertex $t_0$ of the $n$-regular tree $\TT$. Label the arcs in $\sigma$ so that $\sigma$ becomes a labeled triangulation. Then there is a unique way of assigning a labeled triangulation to each vertex of $\TT$ in such a way that for every edge $t\frac{\phantom{xxx}}{k}t'$, the labeled triangulations assigned to $t$ and $t'$ are related by the flip of the arc labeled $k$.

For every (unlabeled) tagged triangulation $\tau$ of $\surf$ there is a path $p_\tau$ in $\mathbf{E}^{\bowtie}\surf$ from $\sigma$ to $\tau$ along whose edges flips are compatible with QP-mutations (this is Theorem \ref{thm:flip<->mutation-tagged-triangs}). This path has a unique lift to a path on $\TT$ starting at $t_0$. The labeled triangulation attached to the terminal vertex of this lift is precisely $\tau$ with some ordering of its elements. By Theorem \ref{thm:admissibility-holds-for-Qtau}, the graph $G$ obtained by lifting all paths $p_\tau$, for $\tau$ unlabeled tagged triangulation, satisfies the hypotheses of Theorem \ref{thm:admissibility=>proper-Laurent-monomials}. The corollary follows.
\end{proof}

\begin{remark} It was shown in \cite{DWZ2} that cluster monomials are linearly independent provided any matrix mutation-equivalent to the initial exchange matrix has full rank. This was shown as well in \cite[Theorem 5.5]{FK}, also under the referred full-rank assumption, for cluster algebras that admit a certain categorification. Corollary \ref{coro:Laurent-for-surfaces} above is stronger, because it removes any rank assumption and because there are quite a number of surfaces whose triangulations do not have full-rank signed-adjacency matrices. Indeed, by \cite[Theorem 14.3]{FST}, the matrix $B(\tau)$ of a triangulation has full rank if and only if the underlying surface $\surf$ is unpunctured and has an odd number of marked points at each boundary component; thus, given a surface with punctures or with at least one boundary component containing an even number of marked points, and given any triangulation $\tau$ of such a surface, one can obtain infinitely many non-full rank extended exchange matrices having $B(\tau)$ as their top square submatrix.
\end{remark}

\begin{remark} An \emph{atomic basis} of a cluster algebra $\mathcal{A}$ is a
$\Z\semifield$-linear basis $\mathcal{B}$ of $\mathcal{A}$ such that the positive
elements are precisely the $\Z_{\geq 0}\semifield$--linear combinations of
elements of $\mathcal{B}$. The existence of atomic basis has been proved only
for a few types of cluster algebras (\cite{DTh},\cite{GCIthesis}, \cite{GCIpaper},\cite{ShZ}). In the case when $\mathcal{A}$
comes from a surface with boundary, Theorem \ref{thm:positive-have-positive-coefficients} and Corollary \ref{coro:Laurent-for-surfaces}
indicate that, if an atomic basis exists, it should contain the cluster monomials.
\end{remark}

\begin{remark} After this paper was submitted have the preprints \cite{CKLP} and \cite{MSWbases} appeared.
\begin{itemize}
\item In \cite{CKLP} it is shown that cluster monomials are always linearly independent in skew-symmetric cluster algebras with arbitrary coefficients, regardless of the rank of the extended exchange matrices. The proof uses an appropriate homological interpretation of the $E$-invariant in generalized cluster categories (not necessarily $\Hom$-finite) that allows to drop the assumption of the admissibility condition and still get the proper Laurent monomial property in skew-symmetric cluster algebras. This appropriate interpretation is based in Plamondon's study \cite{Plamondon-char} \cite{Plamondon-inf} \cite{Plamondon-thesis} of generalized cluster categories whose $\Hom$-spaces are not necessarily finite-dimensional.
\item In \cite{MSWbases}, Musiker-Schiffler-Williams construct two bases for cluster algebras coming from unpunctured surfaces and having full-rank extended exchange matrices. They conjecture that one of these two bases is the atomic basis. Furthermore, in the punctured case they construct two subsets of the corresponding cluster algebra that are good candidates to be bases and contain all cluster monomials.
\end{itemize}
\end{remark}

\section{Atomic bases for types $\mathbb{A}$, $\mathbb{D}$ and $\mathbb{E}$}\label{sec:atomic-bases}

In this section we give an application of Theorem \ref{thm:positive-have-positive-coefficients} to show that cluster monomials form atomic bases in skew-symmetric cluster algebras of finite type.

\begin{thm}\label{thm:atomic-basis-for-finite-type} Let $\myAA$ be a coefficient-free cluster algebra of finite type. Then the set of cluster monomials is an atomic basis of $\myAA$.
\end{thm}

\begin{proof} By \cite[corollary~3]{CK1}, the set of cluster monomials forms a $\Z$-basis of $\myAA$. On the other hand, the cluster monomials of $\myAA$ are positive by results of \cite[Sections 10 and 11]{HL} and \cite[Theorem A.1]{Nakajima} (in types $\mathbb{A}$ and $\mathbb{D}$ this also follows from \cite{MSW}). It remains to show that every positive element of $\myAA$ can be written as a non-negative linear combination of cluster monomials.

Let $Q$ be a quiver mutation-equivalent to an orientation of one of the Dynkin diagrams $\mathbb{A}_n$, $\mathbb{D}_n$, $\mathbb{E}_6$, $\mathbb{E}_7$ and $\mathbb{E}_8$. Then for any two non-degenerate potentials $S$ and $S'$ on $Q$ there is a right-equivalence between $(Q,S)$ and $(Q,S')$. This fact gives us some freedom to choose a suitable potential on $Q$. Let $S$ be the sum of all chordless cycles on $Q$ (cf. \cite[Section 9]{DWZ1}). Then $(Q,S)$ satisfies the admissibility condition.
Therefore, $\myAA$ satisfies the proper Laurent monomial property. Since cluster monomials form a $\Z$-basis of $\myAA$, any positive element is a $\Z$-linear combination of cluster monomials. But then, by Theorem \ref{thm:positive-have-positive-coefficients}, the coefficients in this combination are non-negative.
\end{proof}

We close the paper with a quite intriguing question.

\begin{question} Is the admissibility condition a QP-mutation invariant?
\end{question}

As we saw in Theorems \ref{thm:admissibility=>proper-Laurent-monomials} and \ref{thm:atomic-basis-for-finite-type}, a positive answer to this question would yield desirable properties for the corresponding cluster algebras.

\section*{Acknowledgements}
The collaboration that led to the present paper began during the trimester program \emph{On the Interaction of Representation Theory with Geometry and Combinatorics} held at the Hausdorff Research Institute for Mathematics (HIM, Bonn, Germany) from January to April 2011. Both authors are grateful to the organizers of this program for creating a great long-lasting working atmosphere.

We thank Bernhard Keller for kindly answering the second author's questions on exchange graphs of quivers with potentials and cluster-tilting objects.

Finally, though not less importantly, we express our gratitude to Andrei Zelevinsky, who was Ph.D. thesis advisor of each one of us, for introducing us to cluster algebra theory and its connections with various areas of mathematics. We also thank him for useful comments on the first version of this paper.

\end{document}